\documentclass[12pt,leqno,amscd,amssymb,verbatim, url]{amsart}
\usepackage{amsfonts,amssymb}
\usepackage{amsmath,amscd}
\newtheorem{thm}{Theorem}[section]
\usepackage{amsfonts,latexsym}
\newtheorem{lem}[thm]{Lemma}
\newtheorem{cor}[thm]{Corollary}

\newtheorem{prop}[thm]{Proposition}
\theoremstyle{definition}

\newtheorem{rem}[thm]{Remark}
\newtheorem{rems}[thm]{Remarks}

\numberwithin{equation}{thm}
\usepackage{wrapfig}

\newcommand{\KL}{\text{\rm KL}}

\newcommand\Xreg{X^+_{\text{\rm reg}}}
\newcommand\Xregl{X^+_{\text{\rm reg},l}}

\newcommand{\Xonel}{X_{1,l}^+}

\newcommand{\gen}{{\text{\rm gen}}}

\newcommand{\sC}{{\mathcal {C}}}

\newcommand{\Ext}{{\text{\rm Ext}}}

\newcommand{\St}{\text{\rm St}}

\newcommand{\Hom}{\text{\rm Hom}}

\newcommand{\ch}{\operatorname{ch}}

\newcommand{\soc}{\operatorname{Soc}}

\newcommand{\rad}{\operatorname{rad}}

\newcommand{\rDelta}{\Delta^{\text{\rm red}}}
\newcommand{\rnabla}{\nabla_{\text{\rm red}}}

\newcommand{\Jan}{\Gamma_{\text{\rm Jan}}}

\newcommand{\Gmod}{G{\text{\rm --mod}}}

\newcommand{\opH}{{\text{\rm H}}}

\newcommand{\blist}{\begin{list}{\rom{(\roman{enumi})}}{\setlength
{\leftmarg in}{0em} \setlength{\itemindent}{7ex}
\setlength{\labelsep}{2ex}\setlength{\listparindent}{\parindent}
\usecounter{enumi}}}
\newcommand{\elist}{\end{list}}

\begin{document}

\title[Bounding Ext ]{\large {\bf Bounding Ext for modules
 for algebraic groups, finite groups and quantum groups}}

\author{Brian J. Parshall}
\address{Department of Mathematics \\
University of Virginia\\
Charlottesville, VA 22903} \email{bjp8w@virginia.edu {\text{\rm
(Parshall)}}}
\author{Leonard L. Scott}
\address{Department of Mathematics \\
University of Virginia\\
Charlottesville, VA 22903} \email{lls2l@virginia.edu {\text{\rm
(Scott)}}}
\thanks{Research supported in part by the National Science
Foundation} 

\begin{abstract} Given a finite root system $\Phi$, we show that there is an integer $c=c(\Phi)$ such that
$\dim\Ext_G^1(L,L')<c$, for any reductive algebraic group $G$ with root system $\Phi$ and any irreducible
rational $G$-modules $L,L'$.  There also is such a bound in the case of finite groups of Lie type, depending
only on the root system and not on the underlying field. For quantum groups, a similar
result holds for $\Ext^n$, for any integer $n\geq 0$, using a constant depending only on $n$ and the root system. Weaker versions
of this are proved in the algebraic and finite group cases, sufficient to give similar results for algebraic and generic cohomology. The results both use, and have consequences
for, Kazhdan-Lusztig polynomials. An appendix proves a stable version, needed for small prime arguments, of Donkin's tilting module conjecture. \end{abstract}

\maketitle

\section{Introduction and discussion} In previous work with Edward Cline \cite{CPS7}, we showed that there exists a bound on $\dim\,\opH^1(G, L)$---for a semisimple,
simply connected algebraic group $G$ and an irreducible rational $G$-module $L$---by a constant $C$ that depends only on the
root system of $G$, and not on the module $L$ or the underlying field. We also proved a similar result for the finite groups
of Lie type and irreducible modules in the natural characteristic.\footnote{Recently, Guralnick and Tiep have proved an
analogous result in the cross-characteristic case \cite{GT}.} This result represented the first general progress on a
conjecture of Robert Guralnick \cite{G} that there exists a universal bound on $\dim \opH^1(H,V)$, for all finite groups $H$ and
for all faithful, absolutely irreducible representations $V$ over any field $k$. The space $\opH^1(H,V)$, which parametrizes conjugacy classes of complements to $V$ in the semidirect product
$V\rtimes H$, is particularly important for maximal subgroup theory.\footnote{
See \cite{AsS}, where
$\opH^1(H,V)$ enters into the paramatrization both of maximal subgroups of semidirect products, and larger, more elaborate finite groups. Thus, $\opH^1(H,V)$ is part of a general theory of maximal subgroups. See also \cite[p. 514]{GKKL}, linking $\opH^1(H,V)$ with a conjecture of Wall on maximal subgroups of finite groups. Of course, the conjugacy classes of maximal subgroups, in addition to revealing internal structure, parametrize  the primitive permutation representations of the containing group (through coset actions). }

As is well-known, $\opH^1(H,V)\cong\Ext^1_H(k,V)$, where $k$ is the 1-dimensional trivial module.
For any pair $L,V$ of $kH$-modules, $\Ext^1_H(L,V)$ parametrizes equivalence classes of short
exact sequences $0\to V\to E\to L\to 0$ of $kH$-modules. These sequences are of interest for all $L$, not just for the trivial
module.
Guralnick also expressed the view (privately) that the dimensions of $\Ext^1$-groups between irreducible modules should
also be quite small. We know that no universal constant bound is possible here (see  Scott-Xi \cite{SX}); however,
the asymptotic properties of the growth of these $\Ext^1$-groups has considerable interest.

The first main result of this
paper, see Theorem \ref{MainExtOne theorem}, establishes that, given a finite root system $\Phi$, there
exists a constant $c=c(\Phi)$ such that if $G$ is a semisimple, simply connected algebraic group over an algebraically closed field $k$
of any characteristic $p$ and if $G$ has root system $\Phi$, then $\dim\Ext^1_G(L,L')\leq c$, for all irreducible, rational $G$-modules $L,L'$. This result
improves upon a similar partial result given in \cite[Thm. 7.9]{CPS7}. As in \cite[Thm. 7.10]{CPS7}, it implies the same
result for finite groups of Lie type associated with the root system $\Phi$ with $L,L'$ irreducible modules in the natural
characteristic. See Corollary \ref{finitegroupcorollary}.

In this paper, as in \cite{CPS7}, we attack the
issue of a bound in the algebraic group case by first showing that, if the prime $p$ is fixed, there is a uniform bound for all $L,L'$. Some care is required when $p=2$; see \S3. With this analysis complete,  take $p$ large, and, in particular,
assume that the Lusztig character formula holds. We can then utilize the (Lusztig) quantum group $U_\zeta$ at a $p$th root of 1 associated to $G$.  We also can assume that if $L=L(\lambda)$ and $L'=L(\mu)$ (for dominant
 weights $\lambda,\mu$), then $\lambda\not\equiv
\mu$ mod$\,p$ and $\lambda<\mu$. In \cite{CPS7}, we treated the case for $\lambda,\mu$ regular and showed that
$$\dim\Ext^1_G(L(\lambda),L(\mu))\leq\dim\Ext^1_G(\rDelta(\lambda),\rnabla(\mu))=\dim\Ext^1_{U_\zeta}(L_\zeta(\lambda),L_\zeta(\mu)),$$
using the natural modules $\rDelta(\lambda)$ and $\rnabla(\mu)$ that arise from the irreducible modules $L_\zeta(\lambda)$ and
$L_\zeta(\mu)$, respectively, for $U_\zeta$  by a standard ``reduction mod$\, p$" process. Here, we
treat the singular weight case by focusing more on the (ostensibly larger) group $\Ext^1_G(\Delta(\lambda),L(\mu))$ and by using translation
arguments to reduce to the regular case (an approach announced by us in \cite{S3}), maintaining the condition $\lambda\not\equiv
\mu$ mod$\,p$. The bound is given in that case, as in \cite{CPS7}, by an appropriate
``top'' coefficient $\mu(x,y)$ of a Kazhdan-Lusztig polynomial for the
affine Weyl group of $\Phi$. A different way of treating the case $\lambda\not\equiv\mu$ mod$\,p$ is stated
in Theorem \ref{Gsum}, which bounds an infinite sum of $\Ext^1$-dimensions with a single constant. The proof of this
latter result is postponed to \S7.

Section 6 studies the interaction of quantum group cohomology/representation theory with combinatorial considerations of
Kazhdan-Lusztig polynomials. Let $U_\zeta$ be the (Lusztig) quantum enveloping algebra at an $l>h$ root of unity (where $h$ is the Coxeter
number of $\Phi$). Corollary \ref{weakerbound} shows that $\dim\,\Ext_{U_\zeta}^n(L_\zeta(\lambda),L_\zeta(\nu))$ is bounded uniformly by a constant depending only on $n$ and $\Phi$.
 Actually, a stronger result, given in Theorem \ref{bounding sum of exts}, shows that  there is a uniform bound on the sum of all these dimensions, over all possible dominant weights $\nu$, when the weight $\lambda$ is fixed. The bound again depends only on $n$ and $\Phi$, and not on $\lambda$. In addition to homological applications, there are interesting consequences for Kazhdan-Lusztig polynomials; see Theorem \ref{KL sum bounds}. The key to Theorem \ref{bounding sum of exts} is Lemma \ref{biglengthlemma}, which shows that the composition series length of all PIMs are uniformly bounded in the quantum case, depending only on $\Phi$. While this very quickly yields the boundedness property given in Theorem \ref{bounding sum of exts}, the bounds thus obtained are quite crude.
 (The results of Section 6 lead---in a paper \cite{PS6} in preparation---to a ``complexity theory" for the quantum
 enveloping algebras $U_\zeta$, analogous to the classical cohomological complexity theory for finite groups\footnote{We defer to
 \cite{PS6} some of our previous (posted preprint) remarks relevant
to \cite{GKKL} and possible behavior of higher degree cohomology of finite simple
groups. Related comments in this paper may be found above Corollary \ref{weakerbound}.} and restricted enveloping algebras.)

The main result of \S7, given in Theorem \ref{SecondBigTheorem}, provides a higher degree version of the main $\Ext^1$-result
given in \S5. Specifically, we show that, given nonnegative integers $e,m$, there exists a constant $c(\Phi,m,e)$
depending only on $\Phi$, $e$, and $m$ such that $\dim\Ext^m_G(L(\lambda),L(\mu))\leq c(\Phi,m,e)$, for all $p^{e+1}$-restricted
dominant weights $\lambda$ and all dominant weights $\mu$. The proof is by induction on $m$. When $p\geq 2h-2$, the projective
covers of the irreducible modules for the $r$th infinitesimal subgroup $G_r$ of $G$ are known to have compatible $G$-module structures, and an essential
step involves showing that there is a bound $C^\flat(\Phi,r)$ on the $G$-length of these modules which depends only on
the root system $\Phi$ and the integer $r$; see Lemma \ref{firstlemmatoSecondMainThm}.  As mentioned above, the analogous result for the quantum PIM $Q_\zeta(\lambda)$ for
{\it all} dominant $\lambda$ is established in \S6. The difference between the two cases seems to hinge on the fact that
in the algebraic group case $G/G_1\cong G$,  while in the quantum enveloping algebra case, if $u_\zeta$ is the little quantum
group, then the quotient $U_\zeta/u_\zeta\cong U({\mathfrak g})$ has a completely reducible finite dimensional representation theory.
Another important ingredient makes use of the stability estimates proved in \cite{CPSK} for generic cohomology (viewed in
terms of twists by
powers of of the Frobenius endomorphsm)  of reductive
groups. Section 7 concludes with the postponed proof of Theorem \ref{Gsum}. An interesting consequence of this theorem shows that there are only a finite number of dominant weights $\lambda$
 with  $\lambda\not\equiv 0$ mod $p^2$ for which $H^1(G,L(\lambda))\not=0$.
Further, given any dominant $\mu$, $\dim H^1(G,L(\mu))$ may be computed in terms of some $\dim H^1(G,L(\lambda))$ with $\lambda$ as above. See Remark \ref{concludingremarks}(c). Possibly changing $G$ when $p=2$, we may even take $\lambda\not\equiv 0$ mod $p$.

It is important to mention that the results of Section 7 do not require any restriction on the size of the characteristic $p$.
In an Appendix (Section 8), we develop the necessary machinery to handle the small prime cases. A long open problem in modular representation theory is the existence of a rational $G$-module structure on the PIMs for the infinitesimal subgroups $G_e$ of $G$.
As mentioned above, when $p\geq 2h-2$, the existence of a $G$-module structure is known to hold (by Jantzen and earlier work
by Ballard). On the other hand, for all $p$, Donkin \cite{Donkin2} has conjectured that, for any $e\geq 1$ and any
$p^e$-restricted dominant weight $\lambda$, the $G_e$-PIM $Q_e(\lambda)$ is $G_e$-isomorphic to the restriction to $G_e$ of
the tilting module $T(2(p^e-1)\rho+w_0\lambda)$ of highest weight $2(p^e-1)\rho+w_0\lambda$. We prove, for all characteristics,
a ``stable" version of Donkin's conjecture; see Corollary \ref{directsumcorollary} for a precise
statement. Making use of this result, we construct in Corollary \ref{lastappendixcorollary} a substitute for the $G_e$-PIMs
which can be used in Section 7 to complete the arguments for all characteristics.\footnote{This improves on an earlier posted preprint, which treated
only the case $p>h$ in the presence of the Lusztig character formula.
  Actually, the argument there, using derived category ``shifted
  standard filtrations," appears to be complete only when
  a bound on the exponent of both weights $\lambda$ and $\mu$ is given.
  It does, however, appear to give useful bounds in that case.  }

Our results on algebraic group cohomology for $G$ have the consequence of bounding generic cohomology with irreducible                      coefficients in any given degree $m$ by a constant depending only on $m$ and the root system; see the end of
\S7. If $m>1$, it remains open if there is such a bound for all finite groups of Lie type associated to the root system.
(The issue is that the rate of convergence in the limit $\underset{d\to\infty}\lim H^n(G(p^d),V)$ depends on the coefficient
module $V$.)

           The authors are grateful to Bob Guralnick for comments on an earlier version of this paper.

\section{Some preliminaries}
Let $\Phi$ be an irreducible (classical) finite root system spanning a Euclidean space $\mathbb E$ with inner
product $(u,v)$.  The assumption that $\Phi$ is irreducible is only a convenience, and the setting can easily be
 generalized to the case of a general finite  root system. Fix a set $\Pi=\{\alpha_1,\cdots,\alpha_{\text{rk}(G)}\}$ of simple roots and let $\Phi^+$ be the positive roots determined by
$\Pi$. Let $\alpha_0$ be the maximal short root in $\Phi^+$, and, for $\alpha\in\Phi$, put $\alpha^\vee=\frac{2}{(\alpha,\alpha)}\alpha$, the coroot attached to $\alpha$.    Let $Q=Q(\Phi)$ be the root lattice, i.~e., $Q={\mathbb Z}\alpha_1\oplus\cdots\oplus{\mathbb Z}\alpha_{\text{rk}(G)}
\subset \mathbb E$.

Let $X\subset \mathbb E$ be lattice of all integral weights: $\lambda\in\mathbb E$ belongs to $X$ if and only if $(\lambda,\alpha^\vee)\in\mathbb Z$ for
all $\alpha\in\Pi$. Thus, $Q\subseteq X$. Let $\varpi_1,\cdots,\varpi_{\text{rk}(G)}\in X$ be the fundamental dominant weights defined by $(\varpi_i,\alpha_j^\vee)=\delta_{i,j}$, $1\leq i,j\leq \text{rk}(G)$. Then $X^+:={\mathbb N}\varpi_1\oplus\cdots\oplus{\mathbb N}\varpi_{\text{rk}(G)}$
is the set of dominant weights.
Put $\rho=\varpi_1+\cdots+\varpi_{\text{rk}(G)}$ (the Weyl weight) and $h=(\rho,\alpha_0^\vee)+1$
(the Coxeter number of $\Phi$).

For a positive integer $l$, let $\Xonel:=\{\lambda\in X^+\,|\,(\lambda,\alpha^\vee)<l,\,\forall\alpha\in\Pi\}$ be the set of
$l$-restricted dominant weights. More generally, if $e\geq 1$ is an integer, we set $X^+_{e,l}:=\{\lambda\in X^+\,|\,
(\lambda+\rho,\alpha^\vee)<l^e,\,\,\forall\alpha\in\Pi\}$. When $l=p$, a prime, the set $X^+_{e,p}$ of $p^e$-restricted
dominant weights will be used in Sections 7 and 8.

Regard $\mathbb E$ (or any subset) as a poset by setting $\lambda\leq
\nu$ provided that $\nu-\lambda=\sum_{\alpha\in\Pi}n_\alpha\alpha$, where each $n_\alpha$ is a non-negative integer. Another
partial ordering $\leq'$ is sometimes useful: put $\lambda\leq'\nu$ provided that $\nu-\lambda=\sum_{\alpha\in\Pi}q_\alpha\alpha$,
with each $q_\alpha\in {\mathbb Q}^+$.  An ideal $\Omega$ of $\Gamma$ is a subset such that $\lambda\leq \nu$,
with $\nu\in\Omega$ and $\lambda\in\Gamma$, implies $\lambda\in\Omega$.

The Weyl group $W$ is a Coxeter group with   fundamental reflections $\{s_{\alpha_1},\cdots,s_{\alpha_{\text{rk}(G)}}\}$,
where, given $\alpha\in\Phi$, $s_\alpha:{\mathbb E}\to{\mathbb E}$, $x\mapsto s_\alpha(x):=x-(x,\alpha^\vee)\alpha$, $x\in\mathbb E$. For $\alpha\in\Phi$, $n\in\mathbb Z$, let $s_{\alpha, n}:{\mathbb E}\to{\mathbb E}$ be the affine transformation
$x\mapsto s_{\alpha,n}(x):=x-\left((x,\alpha^\vee)-n\right)\alpha$. Let $W_a$ be the group of affine transformations generated by
the $s_{\alpha,n}$, $\alpha\in\Phi$, $n\in\mathbb Z$.   Since $s_\alpha=s_{\alpha,0}\in W_a$, $W$ is a subgroup of $W_a$; in fact,
$W_a\cong W\ltimes Q$, identifying $Q=Q(\Phi)$ with the subgroup of $\text{Aff}({\mathbb E})$ (= group of affine transformations
of $\mathbb E$) consisting of translations by elements of $Q$. Putting $S_a:= S\cup\{s_{\alpha_0,-1}\}$, $(W_a,S_a)$ is a Coxeter
system.  If $s=s_{\alpha,n}\in W_a$,  $H_s\subset\mathbb E$ is its fixed-point hyperplane.

For a positive integer $l$, let $W_{a,l}$ be the subgroup of $W_a$ generated by the affine reflections $s_{\alpha,n}$ in which $l$ divides $n$.
There is an evident isomorphism $\varepsilon_l:W_a\overset\sim\to W_{a,l}$ in which $s_{\alpha,n}\mapsto s_{\alpha,nl}$.
We will use the ``dot" action of $W_a$ on $\mathbb E$: for $w\in W_a$, $x\in\mathbb E$, put $w\cdot x:=w(x+\rho)-\rho$. For a positive integer $l$, $w\cdot_lx:=\varepsilon_l(w)\cdot x$. Setting $S_{a,l}:=S\cup\{s_{\alpha_0,-l}\}$,
$(W_{a,l},S_{a,l})$ is a Coxeter system.

 For $x,y\in W_{a,l}$, let $P_{y,x}\in{\mathbb Z}[q]$ be the Kazhdan-Lusztig polynomial in $q=t^2$ associated with the
pair $(y,x)\in W_{a,l}\times W_{a,l}$.\footnote{The integer $l$ should always be clear from context when discussing $P_{y,x}$.  Of course, because of the isomorphism $\epsilon_l:W_a=W_{a,1}\overset\sim\to W_{a,l}$, the Kazhdan-Lusztig polynomial $P_{y,x}$, $x,y\in W_{a,l}$, for the Coxeter system $(W_{a,l},S_{a,l})$ identifies with a Kazhdan-Lusztig polynomial $P_{\epsilon_l^{-1}(x),\epsilon_l^{-1}(y)}$
for the Coxeter system $(W_a,S_a)$.}
 Necessarily, $P_{y,x}\not=0$ implies that $y\leq x$ in the Bruhat-Chevalley
order on $W_{a,l}$. If $\ell:W_{a,l}\to\mathbb N$ is the length function (defined by $S_{a,l}$) and $y<x$, then $P_{y,x}$ is a polynomial in $q$ of degree $\leq(\ell(x)-\ell(y)-1)/2$. (Also, $P_{x,x}=1$ for all $x$.) For $y\leq x$, let $\mu(y,x)$ be the coefficient of $q^{(\ell(x)-\ell(y)-1)/2}$ in $P_{y,x}$. Thus,
$\mu(y,x)=0$ unless $x$ and $y$ have opposite parity. In particular, $\mu(x,x)=0$. If $x<y$, set
$\mu(y,x):=\mu(x,y)$, and put $\mu(y,x)=0$ if $x,y$ are not comparable.

 Let $C^-_l\subset\mathbb E$
be the chamber defined by the hyperplanes $H_s$, $s\in\{s_{\alpha_1},\cdots, s_{\alpha_l},s_{\alpha_0,-l}\}$.
Its closure $\overline{C^-_l}$ is a fundamental domain for the dot action of $W_{a,l}$ on $\mathbb E$ (or for the $\cdot_l$-action
of $W_a$ on $\mathbb E$). Let $C^-:=C^-_1$ if $l=1$.

We call $w\in W_a$ dominant if $w\cdot C^- +\rho$ is contained in the dominant cone $\{x\in{\mathbb E}\,|\, (x,\alpha^\vee_i)\geq
0, \, 1\leq i\leq r\}$ of $\mathbb E$. (Similarly, an element $\varepsilon_l(w)=\bar w$ is $l$-dominant if $\bar w\cdot C_l^- +\rho =w\cdot_lC_l^-+\rho$
is contained in the dominant cone. Thus, $w$ is dominant if and only if $\varepsilon_l(w)$ is $l$-dominant.) Let $W_a^+$ (resp.,
$W^+_{a,l}$) be the set of dominant elements in $W_a$ (resp., $W_{a,l}$). For any $x\in W_a$, the right coset $Wx$ of $W$ in
$W_a$ contains a unique element $x'$ of {\it maximal} length among all other elements in the right coset. Then $W_a^+$ is
simply the set of right coset representatives which have maximal length. (Equivalently, $W_a^+$ consists of all elements
 $w_0x$, where $x$ is a ``distinguished" (= minimal length) right coset representative of $W$ in $W_a$, and $w_0\in W$ is the element of
maximal length in $W$.) A similar description of $W_{a,l}^+$ also holds.

Let $\Xregl:=\{\lambda\in X^+\,|\, (\lambda+\rho,\alpha^\vee)\not\equiv 0\,{\text{mod}}\,l,\,\forall\alpha\in\Phi\}$
be the   $l$-regular dominant weights. A weight that is not $l$-regular is called $l$-singular, or just singular. Thus, $\Xregl\not=\emptyset$ if and only if $l\geq h$. Let $\lambda=w\cdot\lambda^-\in X^+$, $\lambda^-\in C^-_l$ and $w\in W^+_{a,l}$.
Define
\begin{equation}\label{LCF}
\chi_{\KL}(\lambda,l)=\sum_{y\in W^+_{a,l}}(-1)^{\ell(w)-\ell(y)}P_{y,w}(1)\chi(y\cdot\lambda^-)\in{\mathbb Z}X.\end{equation}
For $\nu\in X^+$, $\chi(\nu)=\sum_{w\in W}(-1)^{\ell(w)}e(w\cdot\nu)/\sum_{w\in W}(-1)^{\ell(w)}e(w\cdot 0)$ is
the Weyl character in the integer group algebra ${\mathbb Z}X$.   If $l$ is clear from context, write $\chi_{\KL}(\lambda)$ for $\chi_{\KL}(\lambda,l)$.

We work with several algebraic objects attached to the root system $\Phi$.

\smallskip (1) $G$ denotes a simple, simply connected
algebraic group over an algebraically closed field $k$ of positive characteristic with fixed Borel subgroup $B$ containing
a maximal torus $T$.\footnote{Thoughtout this paper, $k$ will always denote an algebraically closed field. The assumption that $G$ be simple is only for convenience. All results  hold if $G$ is only assumed to be semisimple.} The Lie algebras of $G,T,B$, etc. will be denoted by the corresponding fraktur letters $\mathfrak g$,
$\mathfrak b$, $\mathfrak t$, etc. We assume that $\Phi=\Phi(T)$ is the set of roots of $T$, so $X=X(\Phi)=X(T)$ and $Q=Q(\Phi)=Q(T)$.
For $\lambda\in X^+$, let $L(\lambda)$ be the irreducible rational $G$-module of highest weight $\lambda$. Also, let $\Delta(\lambda)$ and $\nabla(\lambda)$
be the standard (Weyl) and costandard modules, respectively, of highest weight $\lambda$. Thus, $L(\lambda)$ is the socle (resp., head) of $\nabla(\lambda)$ (resp.,
$\Delta(\lambda)$). Also, $\Delta(\lambda)$ and $\nabla(\lambda)$ have equal characters given by Weyl's formula, i.~e.,
$\ch\Delta(\lambda)=\ch\nabla(\lambda)=\chi(\lambda)$.

Assume that $p\geq h$. A $p$-regular dominant weight $\lambda=w\cdot\lambda^-$, $\lambda^-\in C^-_p$ is said to satisfy the Lusztig character formula
(LCF) if $\ch\,L(\lambda)=\chi_{\KL}(\lambda,p).$
Recall that the Jantzen region $\Jan:=\{\lambda\in X^+\,|\,(\lambda+\rho,\alpha_0^\vee)\leq p(p-h+2)$.\footnote{It is useful
to note that $\Jan$ contains $X^+_{1,p}$ if and only if $p\geq 2h-3$. Also, if $\sigma=\sigma_0+p\sigma^\dagger\in\Jan$ with
$\sigma_0\in X^+_{1,p}$, $\sigma^\dagger\in X^+$, then $\sigma^\dagger\in \overline C_p$.}
For characteristic $\geq h$ sufficiently large (depending
on $\Phi$) the LCF holds
for all regular $\lambda\in\Jan$; see \cite{AJS} (and also the survey \cite[\S8]{T} for  references) as well as \cite{Fiebig} where the methods of \cite{AJS} are improved to
give a specific bound for each root system. If $l\geq h$, then $\ch L_\zeta(\lambda)=\chi_{KL}(\lambda,l)$ for any $\lambda\in X^+_{\text{reg},l}$. (Suitably formulated, the LCF holds for all $\lambda\in X^+$, $l$-regular or not.) In particular, if $p\geq h$ and if the LCF holds on $\Jan$, then
given $\lambda\in X^+_{\text{reg},l}$, $L(\lambda)$ is obtained  by reduction mod $p$ from $L_\zeta(\lambda)$, $\zeta=\sqrt[p]{1}$.

We  assume that $G$ as well as $B$, $T$ are defined and split over
the prime field ${\mathbb F}_p$. Let $F:G\to G$ be the Frobenius morphism, and, for $e\geq 1$, let $G_e=\ker(F^e)$ be the $e$th infinitesimal subgroup.
If $V$ is a rational $G$-module $V^{(e)}$ denotes the rational $G$-module obtained by pulling the action of $G$ on $V$ back through $G^r$. The set $X^+_{e,p}$ indexes the irreducible rational  $G_e$-modules; given $\lambda\in X_{e,p}^+$, $L(\lambda)|_{G_e}$ is an irreducible $G_e$-module, and representatives
of the distinct isomorphism classes of irreducible $G_e$-modules are given by these modules. (When $L(\lambda)$, or any $G$-module,
is to be regarded as a $G_e$-module byu restriction, we will often be somewhat informal,
writing $L(\lambda)$ instead of $L(\lambda)|_{G_e}$.) For $\lambda\in X^+$, write
\begin{equation}\label{padicexpansion}\lambda=\sum_{i=0}^\infty p^i\lambda_i,\,(\lambda_i\in X_{1,p}^+),\quad \lambda^{(i)}=\sum_{j=i}^\infty p^{j-i}\lambda_j.
\end{equation}
For convenience, we often denote $\lambda^{(1)}$ more simply by $\lambda^\dagger$. Thus, given $\lambda\in X^+$, it has
a unique decomposition $\lambda=\lambda_0+p\lambda^\dagger$, for $\lambda_0\in X^+{1,p},\lambda^\dagger\in X^+$. Also,
$L(\lambda0\cong L(\lambda_0)\otimes L(\lambda^\dagger)^{(1)}$.

\medskip (2) Let $l$ be a positive integer, and let $\zeta=\sqrt[l]{1}$ be a primitive
$l$th root of unity in the complex numbers $\mathbb C$. We will assume that $l$ is an odd integer\footnote{The assumption that $l$ is odd can be avoided, using \cite{Lbook}.} and if $\Phi$ is of type $G_2$, in addition, that $(l,3)=1$.  Let $U_\zeta=U_{(\zeta=\sqrt[l]{1})}$ be the (Lusztig) quantum enveloping algebra over the cyclotomic field ${\mathbb Q}(\zeta)$ with
``root system $\Phi$". In the sequel, when discussing $U_\zeta$, the above restriction on $l$ (imposed by $\Phi$) will always
be in force (though not usually mentioned).
Also, $U_\zeta$-mod denotes the category of finite dimensional $U_\zeta$-modules which are integrable and type $1$. For $\lambda\in
X^+$, let $L_\zeta(\lambda)$ (resp., $\Delta_\zeta(\lambda)$, $\nabla_\zeta(\lambda)$) be the irreducible (resp., standard (Weyl)
costandard) module of highest weight $\lambda$. If $l>h$ and $\lambda\in\Xregl$, then $\ch\,L_\zeta(\lambda)=\chi_{\KL}(\lambda,l)$.\footnote{The
requirement that $\lambda\in \Xregl$ is not necessary, but requires more care in the definition of $\chi_{\KL}(\lambda,l)$
(\ref{LCF}) and
is not needed in this paper. Also, the assumption that $l>h$ can sometimes be relaxed; see \cite[\S7]{T}.}

Let $u_\zeta$ be the ``little" quantum group attached to $U_\zeta$. It is a normal, Hopf subalgebra of $U_\zeta$
such that $U_\zeta//u_\zeta\cong U({\mathfrak g}_{\mathbb C})$ the universal enveloping algebra of the complex simple Lie
algebra with root system $\Phi$. If $M\in U_\zeta$-mod, then the subspace $M^{u_\zeta}$ of $u_\zeta$-fixed points is a
locally finite (and completely reducible) $U({\mathfrak g}_{\mathbb C})$-module (i.~e., a rational module for the complex
algebraic group $G_{\mathbb C}$ with Lie algebra ${\mathfrak g}_{\mathbb C}$.) For $M,N\in U_\zeta$-mod, and
any integer $n$, an elementary Hochschild-Serre spectral sequence argument gives:
\begin{equation}\label{HochSerre} \Ext^n_{U_\zeta}(M,N)\cong \Ext^n_{u_\zeta}(M,N)^{U({\mathfrak g}_{\mathbb C})}.\end{equation}
Let $\text{\rm Fr}:U_\zeta\to U_\zeta//u_\zeta\cong U({\mathfrak g}_{\mathbb C})$  be the quotient (Frobenius) map. Given a
${\mathfrak g}_{\mathbb C}$-module $M$, $M^{(1)}\in U_\zeta$-mod denotes the pullback of $M$ through Fr. For $\lambda\in X^+$,
let $L_{\mathbb C}(\lambda)$ be the irreducible ${\mathfrak g}_{\mathbb C}$-module of highest weight $\lambda$.

\section{Some cohomology results}

Let $G$ be as in \S2(1).
We will need the following result, due to Andersen \cite[Thm. 4.5]{A1}.

\begin{thm}\label{andersentheorem}Unless $p=2$ and $G$ has type $C_r$,
$$\Ext^1_{G_1}(L,L)=0$$
for any rational irreducible $G_1$-module.\end{thm}

Because this result fails when $G$ has type $C_r$ and $p=2$, more attention is required to bound $\Ext^1_G$ in case $p=2$.
Until Proposition \ref{shifting}, $k$ has characteristic $2$ and $G$ (resp., $G'$) is the simple, simply connected algebraic group over $k$ of type $C_r$ (resp., type $B_r$). Let $T'$, etc. be the maximal torus, etc.
of $G'$. The group $G'$ contains a closed subgroup $G^{\prime\prime}$ which is simple of type $D_r$, viz., $G^{\prime\prime}$ is the closed subgroup of $G'$ generated
by $T'$ and the root subgroups $U_\alpha$ corresponding to long roots $\alpha$ in $\Phi'$. The torus $T^{\prime\prime}:=T'$ is a maximal torus in $G^{\prime\prime}$. Also,
$G^{\prime\prime}$ is simply connected because one easily checks that the index $X(T^{\prime\prime})/Q(T^{\prime\prime})$ has order
4.

 The Euclidean space $\mathbb E$ (resp., ${\mathbb E}'$, ${\mathbb E}^{\prime\prime}$) contains $X(T)$ (resp., $X(T')$, $X(T^{\prime\prime})$) and has orthonormal
  basis $\{\epsilon_1,\cdots, \epsilon_r\}$ (resp., $\{\epsilon'_1,\cdots, \epsilon_r'\}$, $\{\epsilon^{\prime\prime}_1,\cdots, \epsilon^{\prime\prime}_r\}$), chosen
  as in  \cite[pp. 267-273]{Bo}. Thus,
\begin{equation}\label{rootsystem}\begin{cases} \Phi=\{\pm \epsilon_i\pm\epsilon_j (i\not= j), \pm 2\epsilon_i \,|\, 1\leq i,j\leq r\},\\
\Phi'=\{\pm \epsilon'_i\pm\epsilon'_j (i\not= j), \pm\epsilon'_i\,|\,1\leq i,j\leq r\},\\
\Phi^{\prime\prime}=\{\pm \epsilon^{\prime\prime}_i\pm\epsilon^{\prime\prime}_j (i\not= j)\,|\,1\leq i,j\leq r\}.\end{cases}
\end{equation}
describe the root systems of $G$, $G'$ and $G^{\prime\prime}$, respectively. Since $T^{\prime\prime}=T'$, we can assume that
$\epsilon_i'= \epsilon_i^{\prime\prime}$.

Identify $X^{\mathbb Q}(T):={\mathbb Q}\otimes X(T)$ with the $\mathbb Q$-span of $\epsilon_1,\cdots,\epsilon_r$ in $\mathbb E$, and make a similar convention
for $X^{\mathbb Q}(T')$ and $X^{\mathbb Q}(T^{\prime\prime})$.
There is a special isomorphism
$\varphi:X^{\mathbb Q}(T')\to X^{\mathbb Q}(T)$, $\epsilon'_i\mapsto 2\epsilon_i, \,1\leq i\leq r$,
as in \cite[Def. 1, \S18.2]{Chevalley}. It corresponds to the bijection\footnote{Note that this bijection
only agrees with $\varphi$ up to  scalar multiples.}
$$ \Phi'\leftrightarrow\Phi,\quad \begin{cases}\pm\epsilon'_i\pm\epsilon'_j\leftrightarrow\pm\epsilon_i\pm\epsilon_j;\\
                      \pm\epsilon'_i\leftrightarrow\pm 2\epsilon_i,\end{cases}$$
for $1\leq i,j\leq r$. For example, observe that
$$\varphi(\varpi'_i)=\begin{cases} 2\varpi_i,\quad 1\leq i<r;\\
\varpi_i,\quad i=r.\end{cases}         \eqno{(*)}            $$
Under this correspondence long (resp., short) roots correspond to short (resp. long) roots.
Also, if $\alpha'\leftrightarrow\alpha$ in (*) with $\alpha'$ short (resp., long), then
$\varphi(\alpha') =\alpha$ (resp., $\varphi(\alpha')=2\alpha$).

The special isomorphism $\varphi$ corresponds to a (special) isogeny $\theta:G\to G'$.  Similarly, there is a (special) isogeny $\theta':G'\to G$ defined by
 the special isomorphism $\varphi':X^{\mathbb Q}(T)\to X^{\mathbb Q}(T')$ which maps each $\epsilon_i$ to $\epsilon'_i$, $1\leq i\leq r$. Regarding $G^{\prime\prime}$ as a subgroup of $G'$, let $\overline{G}_1^{\prime\prime}$ be the scheme-theoretic image
of $G_1^{\prime\prime}$ in the infinitesimal subgroup scheme $G_1$ of $G$ under $\theta'$. Then $\overline{G}^{\prime\prime}_1\cong G_1^{\prime\prime}/\kappa$,
where $\kappa$ is a closed (infinitesimal) subgroup of $T^{\prime\prime}$.  Also, $\overline{G}_1^{\prime\prime}={\text{\rm Ker}}(\theta)$.

 If $\lambda\in\sum_{i=1}^r a_i\varpi_i\in X(T)$, let $\lambda_{\sigma}=\sum_{i=1}^{r-1}a_i\varpi_i\in X(T)$ and $\lambda_\tau=a_r\varpi_r\in X(T)$.
 There exists a unique $\tilde\lambda_\sigma\in X(T')^+$ (resp., $\tilde\lambda_\tau\in X(T')^+$) such that $\varphi(\tilde\lambda_\sigma)=2\lambda_\sigma^{(1)}$
 (resp., $\varphi(\tilde\lambda_\tau)=\lambda_\tau$). Define (somewhat abusing our previous notation)
 \begin{equation}\label{tilde}\tilde\lambda^{(1)}:=\tilde\lambda_\sigma+\tilde\lambda_\tau\in X(T')^+.\end{equation}

Now we return to the general simple group $G$ in the first part of the following result.
\begin{prop}\label{shifting} Let $\lambda,\nu\in X^+$, and suppose that $\lambda_0=\nu_0$.

(a) $\Ext^1_G(L(\lambda),L(\nu))\cong\Ext^1_G(L(\lambda^{(1)}), L(\nu^{(1)}))$ (in the notation of (\ref{padicexpansion}))
unless $p=2$ and $G$ has type $C_r$.

(b) Suppose  $G$ has type $C_r$ and $p=2$.
If $r>2$, then
$$\Ext^1_G(L(\lambda),L(\nu))\cong\Ext^1_{G'}(L(\tilde\lambda^{(1)}),L(\tilde\nu^{(1)})).$$
 \end{prop}

\begin{proof} (a) is proved in \cite[Lem. 7.1]{CPS7}, where it is remarked that it essentially contained in \cite{A1}. We now show (b). If $\lambda_0=\nu_0$, then $\lambda_{\sigma,0}=\nu_{\sigma,0}$. Also, using
\cite[Thm. 11.1]{Steinberg63} as well as the tensor product theorem,
$L(\lambda)=L(\lambda_{\sigma})\otimes L(\lambda_{\tau})$ and $L(\nu)=L(\nu_{\sigma})\otimes L(\nu_\tau)$. Let $M=L(\lambda)^*\otimes L(\nu)$. There is a
Hochschild-Serre exact sequence
\begin{equation}\label{HochSerre}
 0\to \opH^1(\overline{G}^{\prime\prime}_1,M)^{G/{{\overline G}^{\prime\prime}_1}}\to \opH^1(G,M) \to \opH^1(G/{\overline G}^{\prime\prime}_1,M^{{\overline G}^{\prime\prime}_1}).\end{equation}
But $L(\lambda_\tau)$ is a trivial ${\overline G}^{\prime\prime}_1$-module \cite[Thm. 11.1]{Steinberg63}. If
$X:= L(2\lambda_{\sigma}^{(1)})^*\otimes L(\lambda_\tau)^*\otimes L(2\nu_\sigma^{(1)})\otimes L(\nu_\tau),$
then
$ \opH^1({\overline G}^{\prime\prime}_1,M) \cong \Ext^1_{{\overline G}^{\prime\prime}_1}(L(\lambda),L(\nu))
\cong  \Ext^1_{{\overline G}^{\prime\prime}_1}(L(\lambda_{\sigma,0}), L(\nu_{\sigma,0}))\otimes X
=0
$
by Theorem \ref{andersentheorem}, since $\lambda_{\sigma,0}=\nu_{\sigma,0}$, $L(\lambda_{\sigma,0})$ is an irreducible ${\overline G}^{\prime\prime}_1$-module, and $G^{\prime\prime}$ is simply connected of type $D_r$ with $r>2$. Thus,
$
\Ext^1_G(L(\lambda),L(\nu)) \cong \opH^1(G/{\overline G}^{\prime\prime}_1,M^{{\overline G}^{\prime\prime}_1})\cong \opH^1(G',M^{{\overline G}^{\prime\prime}_1}).$
However, $M^{{\overline G}^{\prime\prime}_1}\cong\Hom_{{\overline G}^{\prime\prime}_1}(L(\lambda),L(\nu))\cong \Hom(L(\tilde\lambda^{(1)}),L(\tilde\nu^{(1)}))$, so the result follows.\end{proof}

\section{Connections with quantum enveloping algebras and Kazhdan-Lusztig polynomials}

If $x\in \mathbb E$, the point-stabilizer $(W_a)_x$ for the dot action of $W_a$ on $\mathbb E$ is isomorphic
to a finite parabolic subgroup of $W_a$, so
\begin{equation}\label{maximalorder}
{\text{\rm max}}_{x\in\mathbb E} |(W_a)_x|=|W| <\infty.
\end{equation}
Given $x,y\in W_a$, $\mu(x,y)$ denotes, as in \S2, the coefficient of $q^{(\ell(y)-\ell(x)-1)/2}$ in
the Kazhdan-Lusztig polynomial $P_{x,y}$ for the Coxeter system $(W_a,S_a)$ if $x<y$. If $x>y$, $\mu(x,y):=\mu(y,x)$, and
if $x=y$, then $\mu(x,y)=0$.

\begin{lem}\label{Ebound} There exists a positive integer $E(\Phi)$ such that $\mu(x,y)\leq E(\Phi)$ for all $x,y\in W^+_a$.
\end{lem}

For a representation theory proof, see \cite[Lemma 7.6]{CPS7}, and for a combinatoric proof, see
\cite{SX}. In fact,  \cite{CPS7} shows that
$E(\Phi)=h^{|{\Phi|}}{\mathfrak P}(2h-2\rho)$ works, where $\mathfrak P$ is the Kostant partition function. For another proof in a more general context, see  \S6 below.

 The integers $\mu(x,y)$ in the following result are computed in the Coxeter group $W_{a,l}$.

\begin{lem}\label{quiver}({\text{\rm \cite[(1.5.3)]{CPS7}}}) Given $\lambda,\nu\in \Xregl$ which are $W_{a,l}$-conjugate (under the dot
action), write $\lambda=w\cdot\lambda^-$ and $\nu=y\cdot\lambda^-$ for $\lambda^-\in C^-_l$ and $w,y\in W_{a,l}$. Then
$$\begin{cases}
 \dim\Ext^1_{U_\zeta}(L_\zeta(x\cdot\lambda^-),L_\zeta(y\cdot\lambda^-))
 =\dim\Ext^1_{U_\zeta}(L_\zeta(y\cdot\lambda^-),L_\zeta(x\cdot\lambda^-))
 =\mu(y,x);\\
 \dim\Ext^1_{U_\zeta}(L_\zeta(x\cdot\lambda^-),\nabla_\zeta(y\cdot\lambda^-))=\dim\Ext^1_{U_\zeta}(\Delta_\zeta(y\cdot\lambda^-),
 L_\zeta(x\cdot\lambda^-))=\mu(y,x).\end{cases}.$$
In case $\lambda,\nu$ are not $W_{a,l}$-liked, these $\Ext^1$-groups all vanish.
 \end{lem}

\begin{rem}\label{quantumcoho} More generally, \cite[Thm. 3.5]{CPS1} proves that
$$\begin{aligned} \dim & \Ext^n_{U_\zeta}  (L_\zeta(x\cdot\lambda^-),L_\zeta(y\cdot\lambda^-)) =\\
& \sum_{z\in W^+_{a,l}}\sum_{a+b=n}\dim\Ext^a_{U_\zeta}(L(x\cdot\lambda^-),\nabla_\zeta(z\cdot\lambda^-))\dim\Ext^b_{U_\zeta}(\Delta_\zeta(z\cdot\lambda^-),
L_\zeta(x\cdot\lambda^-)).\end{aligned}$$
  Of course, $\Ext^a_{U_\zeta}(L(x\cdot\lambda^-),\nabla_\zeta(z\cdot\lambda^-))=0$
unless $z\leq x$, and $\Ext^b_{U_\zeta}(\Delta_\zeta(z\cdot\lambda^-),
L_\zeta(x\cdot\lambda^-))=0$ unless $z\leq y$. Letting
$$p_{z,x}:=\sum_{n\geq 0} \Ext^n_{U_\zeta}(L_\zeta(x\cdot\lambda^-),\nabla_\zeta(z\cdot\lambda^-))t^n
=\sum_{n\geq 0}\Ext^n_{U_\zeta}(\Delta_\zeta(z\cdot\lambda^-),L_\zeta(z\cdot\lambda^-))t^n,$$
$t^{\ell(x)-\ell(z)}\bar p_{z,x}$ is the Kazhdan-Lusztig polynomial $P_{z,x}$  (for $W_{a,l}$) in $q=t^2$. (Here $\bar p_{z,x}$ is the polynomial in $t^{-1}$
obtained by replacing each power $t^i$ by $t^{-i}$ in $p_{z,x}$.)
\end{rem}

  Let $G$ be as in \S2(1). For $\lambda\in X^+$, we will make use of four additional $G$-modules: $\rDelta(\lambda)$, $\rnabla(\lambda)$,
  $\Delta^p(\lambda)$, and $\nabla_p(\lambda)$ defined as follows. Let $\lambda=\lambda_0 + p\lambda^\dagger$ as after
    (\ref{padicexpansion}).  For $l\geq h$, the module $\rDelta(\lambda)$ (resp., $\rnabla(\lambda)$)
  is defined to be the reduction modulo $p$ of the $U_\zeta$-irreducible module $L_\zeta(\lambda)$ (for $l=p)$ with respect to
  a minimal (resp., maximal) lattice. When the LCF formula
  holds for all $p$-restricted dominant weights,
  \begin{equation}\label{cline}\begin{cases} \rDelta(\lambda)\cong L(\lambda_0)\otimes\Delta(\lambda^\dagger)^{(1)}; \\
\rnabla(\lambda)\cong L(\lambda_0)\otimes\nabla(\lambda^\dagger)^{(1)}.\end{cases}\end{equation}
    It is not necessary to go into details here.\footnote{ These modules were first introduced by Lusztig \cite{L1}, and then studied
by Lin \cite{Lin}.}
In general, we define
\begin{equation}\label{pweylandpinduced}
\begin{cases} \Delta^p(\lambda)=L(\lambda_0)\otimes\Delta(\lambda^\dagger)^{(1)};\\
\nabla_p(\lambda)=L(\lambda_0)\otimes\nabla(\lambda^\dagger)^{(1)}.\end{cases}\end{equation}
It is easy to see that $\rDelta(\lambda)$ and $\Delta^p(\lambda)$ have head $L(\lambda)$, and
$\rnabla(\lambda)$ and $\nabla_p(\lambda)$ have socle $L(\lambda)$.

Now \cite[Thm. 5.4]{CPS7} gives
\begin{lem}\label{ExtforG} Assume $p>h$ and that the LCF holds for $G$ on an ideal of $p$-regular weights containing the
$p$-regular $p$-restricted weights.

(a) If $\lambda=w\cdot\lambda^-,\nu=y\cdot\lambda^-$, with $\lambda^-\in C^-_p$ and $y\leq w$ in $W_{a,p}$,  then
$$\dim\Ext^1_G(\rDelta(\lambda),\nabla(\nu))=\dim\Ext^1_G(\Delta(\nu),\rnabla(\lambda))=\mu(x,y).$$
Furthermore, if either $\lambda,\nu$ are not $W_{a,p}$-conjugate or if they are conjugate but $y\not\leq w$, then
these $\Ext^1$-groups vanish.

(b) If $\lambda=w\cdot\lambda^-,\nu=y\cdot\lambda^-$, with $\lambda^-\in C^-_p$ and $w,y\in W_{a,p}$, then
$$\dim\Ext^1_G(\rDelta(\lambda),\rnabla(\nu))=\mu(x,y).$$
\end{lem}

\section{Proof of the Main $\Ext^1$-Result}  In this section we first prove the following result.

\begin{thm}\label{MainExtOne theorem} There is, for any finite  root system $\Phi$, a constant
$c=c(\Phi)$ with the following property. If $G$ is a simple, simply connected
algebraic group with root system $\Phi$ over an algebraically closed
field $k$ of arbitrary characteristic $p>0$, then
$$\dim\Ext^1_G(L,L')\leq c$$
for any two irreducible, rational $G$-modules $L,L'$. \end{thm}

It is elementary to reduce to the case in which $G$ is simple, and thus that $\Phi$ is irreducible.
To begin with, if $\Phi$ has type $C_2$ and $p=2$, then we can quote \cite[Prop. 2.3]{Sin} which says that $\dim\Ext^1_G(L(\lambda),L(\nu))\leq 1$ for all
$\lambda,\nu\in X^+$.\footnote{Sin \cite{Sin} gives a precise determination of $\dim\Ext^1_G(L(\lambda),L(\nu))$ for arbitrary $\lambda,\nu$; see also \cite{Ye}.} Thus,
we assume that if $p=2$, then $G$ is not of type $C_2$. Then using Proposition \ref{shifting} repeatedly (if necessary), we need only find a
common bound (depending only on $\Phi$) for the spaces $\Ext^1_G(L(\lambda),L(\nu))$ for $\lambda,\nu\in X^+$ with $\lambda_0\not=\nu_0$.

First, we find a bound for $\Phi$ and the prime $p$ fixed. Since $\lambda_0\not=\nu_0$, the Hochschild-Serre exact sequence (see \ref{HochSerre})
implies that
\begin{equation}\label{tagit}\Ext^1_G(L(\lambda),L(\nu))\cong\Ext^1_{G_1}(L(\lambda),L(\nu))^{G/G_1}.\end{equation}

The following result is proved in \cite[Thm. 7.7]{CPS7}, though the proof there contains some errors.

\begin{lem}\label{Lemma52} Let $\lambda,\nu\in X^+$ satisfy $\lambda_0\not=\nu_0$. Then
$$\dim \Ext^1_G(L(\lambda),L(\nu))\leq p^{|\Phi|}{\mathfrak P}(2(p-1)\rho).$$
\end{lem}

\begin{proof}We can assume that $\lambda<\nu$.
Because $\lambda_0\not=\nu_0$, a simple Hochschild-Serre spectral sequence argument shows that
$$\Ext^1_G(L(\lambda),L(\nu))\cong\Ext^1_{G_1}(L(\lambda),L(\nu))^{G/G_1}.$$
Let $\St=L((p-1)\rho)$  be the Steinberg module. It is self-dual as a rational $G$-module so there exists
a surjection $\St\otimes \St\twoheadrightarrow L(0)=k$ of $G$-modules, and, therefore, tensoring with $L(\lambda)$
and setting $S:=\St\otimes\St\otimes L(\lambda)$, we obtain an exact sequence
\begin{equation}\label{firstseq} 0\to N\to S\to L(\lambda)\to 0\end{equation}
in $\Gmod$.

If $M\in G_1$-mod, let $r_1(M)=\rad_{G_1}(M)$ denote the $G_1$-radical of
$M$. If $M$ is a $G$-module,
then so is $r_1(M)$. In particular, the inclusion $N\hookrightarrow S$ implies that
$r_1(N)\subseteq r_1(S)$. Since $L(\nu)|_{G_1}$ is completely reducible, there are natural maps
 \begin{equation}\label{seqq}
\Hom_{G_1}(S,L(\nu))\overset{\alpha}{\rightarrow}\Hom_{G_1}(N/r_1(N),L(\nu))
\overset{\beta}{\rightarrow}\Hom_{G_1}((r_1(S)\cap N)/r_1(N),L(\nu))\rightarrow
0.\end{equation}
Since $(r_1(S)\cap N)/r_1(N)\subseteq N/r_1(N)$ (and the latter module is completely reducible
as a $G_1$-module), $\beta$ is surjective. Any $G_1$-map $S\to L(\nu)$ vanishes on $r_1(S)$, so
$\beta\circ\alpha=0$. Finally, the cokernel $N/(r_1(S)\cap N)\cong (N+r_1(S))/r_1(S)$ of the
inclusion $(r_1(S)\cap N)/r_1(N)\hookrightarrow N/r_1(N)$ is a $G_1$-direct summand of $S/r_1(S)$.
Thus, any $G_1$-map $N/r_1(N)\to L(\nu)$ vanishing on $(r_1(S)\cap N)/r_1(N)$ lifts to a
$G_1$-map $S/r_1(S)\to N$, so (\ref{seqq}) is exact.
Since $\St|_{G_1}$ is projective, $S|_{G_1}$ is projective, so
$
\Hom_{G_1}(r(S)\cap N,L(\nu))\cong\Ext_{G_{1}}^{1}(L(\lambda),L(\nu))$ by (\ref{firstseq}) and (\ref{seqq}). Taking $G$-fixed
points,  $\Hom_{G}(r(S)\cap
N,L(\nu))\cong\Ext_{G}^{1}(L(\lambda),L(\nu))$.  Thus,
$\dim\,\Ext_{G}^{1}(L(\lambda),L(\nu))\leq\dim S_\nu,$ where $S_\nu$ is the $\nu$-weight space
in $\St\otimes\St\otimes L(\lambda)$. Now repeat the argument in \cite[Lem. 7.6]{CPS7}: if
$\tau$ is a weight in $\St\otimes\St$, then $\lambda<\nu$ implies that
$\dim L(\lambda)_{\nu-\tau}\leq {\mathfrak P}(\lambda-(\nu-\tau))\leq{\mathfrak P}(\tau)\leq{\mathfrak P}(2(p-1)\rho).$
Since $\dim\St=p^{|\Phi^+|}$, we get finally that $\dim S_\mu\leq p^{|\Phi|}{\mathfrak P}(2(p-1)\rho)$.
\end{proof}

At this point, when $\Phi$ and the prime $p$ are fixed, there exists a upper bound
for all the dimensions $\dim\Ext^1_G(L(\lambda),L(\nu))$, $\lambda,\nu\in X^+$.   So, to get a
uniform bound, not depending on $p$, it is enough to treat
uniformly all sufficiently large $p$.

Thus, {\it we
assume $p>h$ and that the LCF holds  for $G$ holds on an ideal of $p$-regular weights containing the $p$-regular $p$-restricted
weights.} We
show the desired bound is $F(\Phi):=|W|E(\Phi)/2$, using the notation of (\ref{maximalorder}) and (\ref{Ebound}).

 Lemmas \ref{Ebound} and \ref{ExtforG} imply that
$\dim\Ext^1_G(\Delta(\lambda),\rnabla(\nu))\leq E(\Phi),$
when $\lambda,\nu\in\Xreg$. If $\lambda,\nu\not\in\Xreg$, we can assume $\lambda=w\cdot\lambda^-$ and $\nu=y\cdot\lambda^-$, for some
$\lambda^-\in\overline{C^-_p}$. Let $T_a^b$ be the translation operator from the category of
$G$-modules with irreducible $W_{a,p}$-conjugate to $\lambda^-$ to the category of $G$-modules
with highest weight linked to $-2\rho$. Let $T^a_b$ be its (left or right) adjoint. Then $L(\nu_0)=T^a_bL(\tau_0)$, for some
restricted $p$-regular weight $\tau_0$.\footnote{Since $\nu_0\in X^+_1$, $0<(\nu_0+\rho,\alpha^\vee)\leq p$, for $\alpha\in\Pi$. Thus, if $\nu_0$ is in the  upper closure of an alcove containing $\sigma\in X^+$, then $0<(\sigma+\rho,\alpha^\vee)<p$, for
all $\alpha\in\Pi$. This means that $\sigma=\tau_0\in X_1^+$.} A standard argument (see \cite{PS5}) shows that $\rnabla(\nu)=
T^a_b\rnabla(\tau)$, where $\tau=\tau_0+p\nu^\dagger$. Thus,
$\Ext^1_G(\Delta(\lambda),\rnabla(\nu))\cong\Ext^1_G(\Delta(\lambda),T^a_b\rnabla(\tau))\cong\Ext_G^1(T^b_a\Delta(\lambda),
\rnabla(\tau)).$
But $T^b_a\Delta(\lambda)$ has a $\Delta$-filtration with sections of the form $\Delta(v\cdot\lambda')$, $v\in (W_{a,p})_x$ and
$\lambda\in X^+_{\text{reg},p}$, where $x$ belongs to the facet
containing $\lambda^-$. Each $\Delta(v\cdot\lambda')$ occurs  with multiplicity at most $1$, and some $p$-regular $\lambda'$. So
$\dim\Ext^1_G(\Delta(v\cdot\lambda'),\rnabla(\tau))\leq E(\Phi),$
by the above. Also, half of these sections satisfy
$\dim\Ext^1_G(\Delta(v\cdot\lambda'),\rnabla(\tau))=\dim\Ext^1_{U_\zeta}(\Delta_\zeta(x\cdot\lambda'),L_\zeta(\tau))=0,$
for those $v\in (W_{a,p})_x$ which have the same parity as $\nu$. (The group $(W_{a,p})_x$ is generated by reflections.\footnote{Here we are identifying regular weights with
 elements of $W^+_{a,p}$, and ``parity" refers to the parity of the corresponding Coxeter group elements. Recall that
$\mu(x,y)=0$ unless $x,y$ have opposite parity.})
Thus,
$\dim\Ext_G^1(T^b_a\Delta(\lambda),\rnabla(\nu))\leq (|(W_{a,p})_x|/2)E(\Phi)\leq ({|W|/2})E(\Phi)=F(\Phi)$.

We can assume that $\lambda,\nu\in X^+$ have distinct $p$-restricted parts (i.~e., $\lambda_0
\not=\nu_0$). Consider the following diagram
$$\begin{CD} \Ext^1_G(L(\lambda),L(\nu)) @. \\
@V\alpha VV @. \\
\Ext^1_G(\Delta(\lambda),L(\nu)) @>>\beta> \Ext^1_G(\Delta(\lambda),\rnabla(\nu))
\end{CD}$$
where $\alpha$ is induced by the
surjection $\Delta(\lambda)\twoheadrightarrow L(\lambda)$ and $\beta$ is induced by the injection
$L(\nu)\hookrightarrow\rnabla(\nu)$. By the long exact sequence of $\Ext^\bullet_G$,
the kernel of $\beta$ is an image of $\Hom_G(\Delta(\lambda),\rnabla(\nu)/L(\nu))$. However, $\rnabla(\nu)\cong L(\nu_0)\otimes\nabla(\nu^\dagger)^{(1)}$, so all
the composition factors of $\rnabla(\nu)$ have the form $L(\nu_0)\otimes L(\tau)^{(1)}
\cong L(\nu_0\otimes p\tau)$ for some $\tau\in X^+$. Since $\lambda_0\not=\nu_0$,
there are no nonzero homomorphisms $\Delta(\lambda)\to \rnabla(\nu)/L(\nu)$, so $\beta$ is an injection. Similarly, the kernel of $\alpha$ is an image of
$\Hom_G(\rad\Delta(\lambda),L(\nu))$, which is also $0$ since $\lambda<\nu$. Thus,
$\alpha$ is  an injection. Hence,
$\dim \Ext^1_G(L(\nu),L(\lambda))=\dim\Ext^1_G(L(\lambda),L(\nu))\leq\dim\Ext^1_G(\Delta(\lambda),\rnabla(\nu))
\leq F(\Phi).$
This completes the proof the Theorem \ref{MainExtOne theorem}

\medskip
Just as in \cite[Thm. 7.10]{CPS7}, the following result for finite groups holds.

\begin{cor}\label{finitegroupcorollary} There is a constant $c'=c'(\Phi)$ with the following property. Let $\sigma:G\to G$ be an endomorphism such
that the group $G_\sigma$ of $\sigma$-fixed points is a finite group. Then $\dim \Ext^1_{G_\sigma}(L,L')<c'$ for all
irreducible $G_\sigma$-modules $L,L'$ over $k$.\end{cor}

Finally, we state the following further $\Ext^1$-results. The proof will be given in \S7.

\begin{thm}\label{Gsum}  There exists a constant $\widetilde C(\Phi)$ depending only on $\Phi$ such that
if $G$
is any simple, simply connected algebraic group, then, for any $\lambda\in X^+$,
 $$\sum_{\nu\,:\,\nu_0\not=\lambda_0}\dim\Ext^1_G(L(\lambda),L(\nu))<\widetilde C(\Phi).$$
\end{thm}

\begin{cor}\label{cortoGsum}  For every dominant weight $\lambda$, there  are at most $\widetilde C(\Phi)$
 dominant weights $\mu$ with $\mu_0\not=\lambda_0$ and $\Ext^1_G(L(\lambda),L(\mu))\not=0$.\end{cor}

There is no bound when $\mu_0=\lambda_0$. However, one can always reduce to the
case of $\mu_0\not=\lambda_0$ as in the proof of Theorem \ref{MainExtOne theorem}. See Remark \ref{concludingremarks}(c) for a discussion of the $\lambda=0$ case.

\section{Further Kazhdan-Lusztig theory} Throughout this section, $\Phi$ is an irreducible root system as in \S2. A proof of
 the following very elementary result is left to the reader.

\begin{lem}\label{boundinglemma} If $\lambda,\nu,\tau\in X^+$, the multiplicity $[L_{\mathbb C}(\lambda)\otimes L_{\mathbb C}(\nu):L_{\mathbb C}(\tau):L_{\mathbb C}(\tau)]$ is
at most $\dim\,L_{\mathbb C}(\tau)$. Also, the inequality
$\text{\rm length}(L_{\mathbb C}(\lambda)\otimes L_{\mathbb C}(\nu))\leq \text{\rm min}\left[\dim\,L_{\mathbb C}(\lambda),
\dim\,L_{\mathbb C}(\nu)\right]$ on lengths holds.\end{lem}

This result fails in positive characteristic, e.~g., let $G=SL_2(k)$ for  characteristic $p>0$. Identify
$X^+$ with $\mathbb N$. For $r\geq 0$, put $V(r)=L(1)\otimes \St(r)$, where $\St(r):=\bigotimes_{i=0}^rL(p-1)^{(i)}$
is irreducible. Since $L(1)$ and
$L(p-1)$ are isomorphic to the Weyl modules $\Delta(1)$ and $\Delta(p-1)$, respectively, $L(1)\otimes L(p-1)$ has a Weyl filtration with sections $\Delta(p)$ and $\Delta(p-2)\cong L(p-2)$ (Clebsch-Gordan). Thus, $V(r)$ has $V(r-1)^{(1)}$ and
the irreducible module $S(r):=L(p-2)\otimes L(p-1)^{(1)}\otimes\cdots\otimes L(p-1)^{(r)}$ as subquotient modules.
Continuing, we see $V(r)$ has $S(1)^{(r-1)}, S(2)^{(r-2)},\cdots, S(r)$ among its irreducible composition factors.
Now let $r\to\infty$.

\medskip
 If $x,y\in W_a$, and
$m\in\mathbb N$, let
$$c_{x,y}^{[m]}=\,{\text{\rm coefficient of}}\,t^m\,\,{\text{\rm in the Kazhdan-Lusztig polynomial}}\, P_{x,y}\,{\text{\rm for}}\, W_a.$$ Thus, $c_{x,y}^{[m]}=0$
unless $x\leq y$ in the partial ordering on $W_a$. If $x< y$, then $P_{x,y}\in{\mathbb N}[t]$ is a polynomial in $t$ of (even) degree
$\leq \ell(y)-\ell(x)-1$, $c^{[m]}_{x,y}=0$ unless $0\leq m\leq \ell(y)-\ell(x)-1$.

The following result is a weak version (and corollary) of Theorem \ref{KL sum bounds}. However, its proof here is quite different and is potentially
useful.
\begin{thm}\label{bounding KL coefficients} Let $m\in\mathbb N$.  There exists an integer $d(\Phi,m)$ such that if $x,y\in W^+_a$ with $x\leq y$, then
$$c_{x,y}^{[\ell(y)-\ell(x)-m]}\leq d(\Phi,m).$$
\end{thm}

\begin{proof}Pick an integer $l> h$, which is odd and not divisible by 3 if $\Phi$ has type $G_2$, and let $U_\zeta$ be  as in \S2(2).  Using the isomorphism $\epsilon:W_a\overset\sim\to W_{a,l}$, it suffices to prove the result for
$W^+_{a,l}$---that is, we can assume that $c_{x,y}^{[\ell(y)-\ell(x)-m]}$ is a coefficient in the
Kazhdan-Lusztig polynomial $P_{x,y}$ for $W^+_{a,l}$. Let $C_1,\cdots, C_s$ be the $l$-restricted alcoves
in $\mathbb E$. For $1\leq i\leq s$, let $\lambda_i\in C_i$ be the unique
dominant weight $W_{a,l}$-linked to $-2\rho\in C_l^-$. Let $M=\bigoplus_{i=1}^s L_\zeta(\lambda_i)$, and consider the
$U({\mathfrak g}_{\mathbb C})=U_\zeta//u_\zeta$-module $M':=\Ext^m_{u_\zeta}(M,M)^{(-1)}$. For dominant weights $\lambda=\lambda_0+l\lambda^\dagger$ and $\nu=\nu_0+l\nu^\dagger$ ($\lambda_0,\nu_0\in X_1^+$, $\lambda^\dagger,\mu^\dagger\in X^+$),
$$\dim\Ext^m_{U_\zeta}(L_\zeta(\lambda),L_\zeta(\nu))\leq\dim \Hom_{U({\mathfrak g}_{\mathbb C})}(L_{\mathbb C}(\lambda^\dagger)\otimes
L_{\mathbb C}(\nu^{\dagger\star}),M')\leq \dim M'$$
by Lemma \ref{boundinglemma}, putting $\nu^{\dagger\star}:=-w_0\nu^\dagger$.
Now let $\nu=x\cdot(-2\rho)$ and $\lambda=y\cdot(-2\rho)$. Then
$$ c^{[\ell(y)-\ell(x)-m]}_{x,y}=\dim \Ext^m_{U_\zeta}(\Delta_\zeta(\nu),L_\zeta(\lambda))\leq\dim\Ext^m_{U_\zeta}(L_\zeta(\nu),L_\zeta(\lambda))
\leq \dim M'$$
by Remark \ref{quantumcoho}.
So $d(\Phi,m):=\dim M'$ works.\end{proof}

We now work with $U_\zeta$ where $\zeta=\sqrt[l]{1}$  as per \S2(2). For $\lambda\in X^+$, let $Q_\zeta(\lambda)$ be the
projective cover of $L_\zeta(\lambda)$ in $U_\zeta$-mod.

\begin{lem}\label{biglengthlemma}  There is a constant $C(\Phi)$, such that, given any $\lambda\in X^+$,
 $${\text{\rm length}}\left(Q_\zeta(\lambda)\right)\leq C(\Phi)$$
 for any quantum enveloping algebra $U_\zeta$ of type $\Phi$ for $l$ odd,  not divisible by 3 in case $\Phi$ has type $G_2$, and otherwise arbitrary. \end{lem}

 \begin{proof} We first show that there is a constant bounding the length of
 any $Q_{\zeta}(\lambda_0)$ for $\lambda_0\in X^+_{\text{reg},l}\cap X^+_{1,l}$ for
 all $l$.
 For fixed $l$,  $|X_{1,l}^+|<\infty$, so it suffices to give
 is a bound that works universally for all $l\geq h$.  It is known that $Q_\zeta(\lambda_0)$ has highest weight $2(l-1)\rho+w_0\lambda_0$. For $\nu\in \Xregl$, the multiplicity
 of $\Delta_\zeta(\nu)$ as a section in a $\Delta_\zeta$-filtration of $Q_\zeta(\lambda_0)$ equals, by Brauer-Humphreys
 reciprocity, the multiplicity
$[\nabla_\zeta(\nu):L_\zeta(\lambda_0)]=[\Delta_\zeta(\nu):L_\zeta(\lambda_0)].$
 If this multiplicity $\not=0$, necessarily $\nu\in\Xregl$ and $\nu\leq 2(l-1)\rho+w_0\lambda_0\leq 2(l-1)\rho$.
 Thus, the number of possible $\nu$
 is absolutely bounded by some integer independent of $l$.  But
 $[\Delta_\zeta(\nu):L_\zeta(\lambda_0)]$ is expressed in terms of the coefficients of inverse
 Kazhdan-Lusztig polynomials $Q_{y,x}$, $x,y\in W_{a}$ satisfying $x\cdot_l\lambda^-=\lambda_0$ and $y\cdot_l\lambda^-=\nu$,
 $\lambda^-\in C^-_l$.  (For a discussion of the $Q_{x,y}$, see, e.~g., \cite[\S7.3]{DDPW}.) Since, independently of $l$, there are only a finite number of possible $x,y\in W_a$, these multiplicities are also bounded independently of $l$.

 Suppose that $\lambda=\lambda_0+l\lambda^\dagger$ as after \ref{padicexpansion}.  Then
 $Q_\zeta(\lambda)=Q_\zeta(\lambda_0)\otimes L_{\mathbb C}(\lambda^\dagger)^{(1)}$. If $\nu=\nu_0+l\nu^\dagger\in \Xregl$ is so that
 $L_\zeta(\nu)$ is a composition factor of $Q_\zeta(\lambda_0)$, then $\nu_0+l\nu_1\leq 2(l-1)\rho+w_0\lambda_0\leq (l-1)\rho$.
 Thus, $\nu_1\leq'\rho$, so that there is a bound on the possible  $\dim\, L_{\mathbb C}(\nu_1)$. By Lemma \ref{boundinglemma},
 this integer bounds the number of composition factors $L(\tau)$ of $L_{\mathbb C}(\nu_1)\otimes L_{\mathbb C}(\lambda^\dagger)$.
 For such a $\tau$, $L_\zeta(\nu_0)\otimes L_{\mathbb C}(\tau)^{(1)}$ is an irreducible $U_\zeta$-module. Hence, there is an absolute bound on the number of composition factors of any $Q_\zeta(\lambda)$.

 Thus, the result holds for $Q_\zeta(\lambda)$ with  $\lambda\in\Xregl$. However, if $\lambda\notin\Xregl$, then $Q_\zeta(\lambda)$ is a direct summand of the translate of some $Q_\zeta(\lambda^\#)$, with $\lambda^\#\in\Xregl$.
 Since translation operators from $l$-regular weights to $l$-singular weights preserve irreducible modules (or map them
 to zero), the length of $Q_\zeta(\lambda)$ is bounded by the length of $Q_\zeta(\lambda^\#)$, and the result is completely
 proved.
 \end{proof}

The following is an immediate consequence.
\begin{cor}  For any $\lambda\in X^+$,
${\text{\rm length}}(\Delta_\zeta(\lambda))\leq C(\Phi)$
for the standard modules $\Delta_\zeta(\lambda)$ for any quantum group $U_\zeta$ of type $\Phi$.\end{cor}

We next have the following application. The ``sum formulation" here and in Theorem \ref{Gsum} is inspired by a somewhat
analogous
use of sums in \cite{GT}.

\begin{thm}\label{bounding sum of exts} For a fixed $n$, there is a constant $C'(\Phi,n)$ such that, for all $\lambda\in X^+$,
 \begin{equation}\label{boundedsum}
 \sum_\nu\dim\Ext^n_{U_\zeta}(L_\zeta(\lambda),L_\zeta(\nu))\leq C'(\Phi,n)\end{equation}
 for any quantum group $U_{\zeta}$ of type $\Phi$ ($l$ arbitrary).
\end{thm}

 \begin{proof}If $P_\bullet\twoheadrightarrow L_\zeta(\lambda)$ is a minimal projective resolution,
 $\dim\Ext^n_{U_\zeta}(L_\zeta(\lambda),L_\zeta(\nu))=\dim\Hom_{U_\zeta}(P_n,L_\zeta(\nu))$
  equals the number of times $Q_\zeta(\nu)$ appears as a direct summand of $P_n$. The number of
 indecomposable summands of $P_0=Q_\zeta(\lambda)$ is 1. For $P_1$ the number
  of indecomposable summands is (strictly) bounded by $\text{length}(P_0)\leq C(\Phi)$. Then the number of indecomposable summands of $P_2$ is bounded by
 $C(\Phi)^2$, \dots, and, finally, the number of indecomposable summands of $P_n$ is bounded by $C(\Phi)^n$. Thus,
 $C'(\Phi,n)=C(\Phi)^n$ works.\end{proof}

\begin{rem}\label{bigremark} We briefly indicate further results which can be found in \cite{PS6} and depend upon
\cite{PS4}.  For regular domiant weights $\lambda=x\cdot\lambda^-,\nu=y\cdot\lambda^-$, write $\mu(\lambda,\nu):=\mu(x,y)$.
By Theorem \ref{bounding sum of exts}, there are only
finitely many $n$-tuples $(\lambda_1,\cdots,\lambda_{n-1},\nu)$ of dominant weights for which the non-negative integers $\mu(\lambda,\lambda_1), \cdots, \mu(\lambda_{n-1},\nu)$
are all nonzero. Also, the dimensions of the $\Ext^1_{U_\zeta}(L_\zeta(\lambda_i),L_\zeta(\lambda_{i+1}))$ (with $\lambda_0=\lambda$ and $\lambda_n=\nu$ and $0\leq i\leq n-1$) are all uniformly
bounded by an integer independent of the weights and $l$. Thus, the right-hand side of  $$\sum_\nu\dim\Ext^n_{U_\zeta}(L_\zeta(\lambda),L_\zeta(\nu))\leq\sum_{(\lambda_1,\cdots,\lambda_{n-1},\nu)}\mu(\lambda,\lambda_1)\mu(\lambda_1,\lambda_2)
\cdots\mu(\lambda_{n-1},\nu)<\infty.$$

 This discussion suggests the question of determining
$$R_\Phi:=\text{\rm Max}_{x\in W_a^+}\left( \sum_{y\in W_a^+} \mu(x,y)\right).$$
Theorem \ref{bounding sum of exts} implies this maximum is finite depending on $\Phi$, but the
argument does not give a good bound (which remains an open problem).  Theorem \ref{bounding sum of exts} gives an
exponential bound
$\sum_\nu\dim\Ext^n_{U_\zeta}(L_\zeta(\lambda),L_\zeta(\nu))\leq R_\Phi^n.$
   However, dropping  the sum over $\nu$, \cite{PS6} gives a polynomial growth bound $\dim\Ext^n_{U_\zeta}(L_\zeta(\lambda),L_\zeta(\nu))\leq D(\Phi)n^{|\Phi|-1},$
with $D(\Phi)$ a constant depending on $\Phi$, but not on $\lambda$ or $\nu$.

Determination of $\sum_{y\in W_a^+} \mu(w_0,y)$
(or a good bound for it) is an open problem, related to bounding $1$-cohomology (and the Guralnick conjecture).
It is currently open whether $\mu(w_0,y)$ is bounded over all $\Phi$, with $3$ the largest
known value; see \cite{S2}. By \cite{SX},  $\mu(x,y)\to\infty$  with larger (type $A$) root systems.
In particular, the constant $D(\Phi)$ must depend on $\Phi$ and tend to infinity as $\Phi$ gets large.
Conceivably, in the spirit of Guralnick's conjecture, one might replace $D(\Phi)$ by a universal constant
if $\lambda$ is fixed as $\lambda=0$ and $\mu$ is allowed to be arbitrary.
\end{rem}

 \begin{cor}\label{weakerbound}For a fixed $n$, there is a constant $C(\Phi,n)$ such that,  for all $\lambda,\nu\in X^+,$
 $$\dim\Ext^n_{U_\zeta}(L_\zeta(\lambda),L_\zeta(\nu))\leq C(n,\Phi)$$
 for any quantum group $U_{\zeta}$ of type $\Phi$ ($l$ arbitrary).\end{cor}

\begin{cor}\label{doubleprimecor} There is a constant $C^{\prime\prime}(\Phi,n)$ for any $n$ such that, for any $\lambda\in X^+$,
 $$\sum_\nu\dim\Ext^n_{U_\zeta}(L_\zeta(\lambda),\nabla_\zeta(\nu))\leq C^{\prime\prime}(\Phi,n)$$
  for any quantum group $U_{\zeta}$ of type $\Phi$ ($l$ arbitrary).\end{cor}

  \begin{proof} By adjoint associativity of translation functors, it suffices to consider only $l$-regular weights. By Remark \ref{quantumcoho},
  $\dim\Ext^n_{U_\zeta}(L_\zeta(\lambda),\nabla_\zeta(\nu))\leq\dim\Ext^n_{U_\zeta}(L_\zeta(\lambda),
  L_\zeta(\nu))$. Now apply Theorem \ref{bounding sum of exts}. \end{proof}

Using this corollary and Remark \ref{quantumcoho} again, we get
\begin{thm}\label{KL sum bounds} Let $m\in\mathbb N$. Let $C^{\prime\prime}(\Phi,m)$ be as in Corollary \ref{doubleprimecor}. If $y\in W_a^+$, then
$$\sum_{x\leq y, x\in W_a^+} c_{x,y}^{[\ell(y)-\ell(x)-m]}\leq C^{\prime\prime}(\Phi,m).$$
\end{thm}

 \section{Higher $\Ext^n$ for algebraic groups} We first prove a higher degree version of Theorem
 \ref{MainExtOne theorem}. For $\nu\in X^+$, let $e_p(\nu)$ denote   the exponent of the largest power of $p$
  appearing in
 the $p$-adic expansion $\nu$. Equivalently, $e_p(\nu)$ is the smallest nonnegative integer $e$ such
 that $\nu\in X^+_{e+1,p}$, i.~e., if $\nu=\sum a_i\varpi_i$, then each $a_i<p^{e+1}$.

 \begin{thm}\label{SecondBigTheorem} Let $m,e$ be nonnegative integers. There exists a constant $c(\Phi,m,e)$ with the following property. If $G$ is a semisimple, simply connected
algebraic group with root system $\Phi$ over an algebraically closed field $k$ of characteristic $p$, then,
for $\lambda,\nu\in X^+$ with $e_p(\lambda)\leq e$, $$\dim\Ext^m_G(L(\lambda),L(\nu))=\dim\Ext^m_G(L(\nu),L(\lambda))\leq c(\Phi,m,e).$$
 In particular,
 $\dim \opH^m(G,L(\nu))\leq c(\Phi,m,0),\quad\forall\nu\in X^+.$
\end{thm}

We can assume to start that $G$ is simple, i.~e., it is as in \S2(1). The proof requires two lemmas.
Given $e\geq 1$ and $\tau\in X^+_{s,p}$, let $Q_e(\tau)$ be the projective cover of the irreducible
$G_e$-module $L(\tau)|_{G_e}$. It is known that $Q_e(\tau)$ is the injective hull of $L(\tau)|_{G_s}$.
When $s=1$, so that $\tau\in X^+_{1,p}$, it will be sometimes convenient to denote $Q_1(\tau)$ by $Q_1(\tau)$.
When $p\geq 2h-2$, each $Q_e(\tau)$, $e\geq 1, \tau\in X^+_{e,p}$,  has a compatible structure as a rational $G$-module \cite[\S 11.11]{Jan}. In that case,
writing $\tau=\tau_0+p\tau_1+\cdots+p^{s-1}\tau_{e-1}$ as per (\ref{padicexpansion}),
 $Q_e(\tau)\cong Q_1(\tau_0)\otimes Q_1(\tau_1)^{(1)}\otimes\cdots\otimes Q_1(\tau_{e-1})^{(e-1)}$ as rational $G$-modules.
In addition, $Q_e(\tau)$ has highest weight $2(p^e-1)\rho+w_0\tau)$. (This later statement is true for all $p$, if
$Q_e(\nu)$ is regarded as a $G_eT$-module.)

\begin{lem}\label{firstlemmatoSecondMainThm}
Let $f$ be a positive integer. There exists a constant $C^\flat(\Phi,f)$ satisfying
the following condition. Let $G$ be a simple, simply connected algebraic group, having root system $\Phi$,  over an algebraically closed field $k$ of
characteristic $p\geq 2h-2$. If $\nu\in X^+_{f,p}$, then
$\text{\rm length}(Q_f(\nu))\leq C^\flat(\Phi,f)$.  In addition, if $L(\omega)$ is a composition factor of $Q_f(\nu)$,
then $e_p(\omega)\in X^+_{f+1,p}$.
\end{lem}

\begin{proof} We prove the last assertion first.  Let $L(\omega)$ be a composition factor of $Q_f(\nu)$, so that
$\omega\leq 2(p^f-1)\rho+w_0\nu)$. For any simple root $\alpha$, $(\omega,\alpha^\vee)\leq (\omega,\alpha_0^\vee)
\leq (2(p^f-1)\rho,\alpha_0^\vee)=2(p^f-1)(h-1)\leq (p^f-1)p<p^{f+1}$. Therefore, $\omega\in X^+_{f+1,p}$.

For a given prime $p$, there are only finitely many $\nu\in X^+$ satisfying $e_p(\nu)<f$, and hence only finitely
many modules $Q_f(\nu)$, which collectively have a bounded length. Therefore, we need to find a uniform bound for the
$Q_f(\nu)$ under the assumptions that $p\geq 3h-3$ and that
the LCF holds in the Jantzen
region $\Jan$. These assumptions
will remain
in effect
for the remainder
of the proof. (In that case, we will give an explicit formula for $C^\flat(\Phi,f)$ in terms of the bound
 $C(\Phi)$ of Lemma \ref{biglengthlemma}.) Let $\zeta=\sqrt[p]{1}$. For $\tau\in X^+_{1,p}$,
$Q_1(\tau)$ is obtained by ``reduction mod $p$" from the $U_\zeta$-projective indecomposable module $Q_\zeta(\nu)$.\footnote{Here
 $Q_1(\tau)$ is the projective $G$-module in the full subcategory $\sC'$ of $\Gmod$ with objects having composition factors $L(\nu)$
with $\nu\leq 2(p-1)\rho+w_0\tau$. Thus, it is the reduction mod $p$ of some $U_\zeta$-module $Q'_\zeta(\lambda_0)$ by \cite[\S3]{DS}, itself projective
in an analogous category. However, it is easy to argue, from the validity of the LCF, together with the assumption that
 $p\geq 3h-3$ that $\dim Q_1(\tau)=\dim Q_\zeta(\tau)$ for each $\tau\in X_{1,p}^+$. Also, $Q_\zeta(\lambda_0)$ is
  also projective in the quantum version $\sC_\zeta$ of $\sC$. We conclude $Q_\zeta(\lambda_0)=Q'_\zeta(\lambda_0)$.} By
Lemma \ref{biglengthlemma}, $Q_\zeta(\tau)$ has length at most $C(\Phi)$. The validity of the LCF implies that each composition factor of $Q_\zeta(\tau)$ reduces mod $p$
to an irreducible $G$-module. Therefore, ${\text{length}}(Q_\zeta(\tau))=\text{length}(Q_1(\tau))$.

We claim that, given a non-negative integer $e$, there is a positive integer $C^\diamond(\Phi,e)$ such that if
$e_p(\nu)\leq e$, then $\text{length}(\Delta(\nu))\leq C^\diamond(\Phi,e)$. If $e=0$, then $e_p(\nu)\leq e$ means that $\nu\in X_{1,p}$.
Then $\text{length}(\Delta(\nu))\leq{\text{length}}(Q_1(\nu))\leq C(\Phi)$ by the above paragraph. We prove the claim by induction on $e>0$. Assume that $e_p(\nu)= e$. Since $\Delta(\nu)$ can be realized by reduction mod $p$ from
$\Delta_\zeta(\nu)$, $\Delta(\nu)$ has a filtration with at most $\text{length}(\Delta_\zeta(\nu))\leq C(\Phi)$ sections
each having character $\chi_{KL}(\tau)=\ch(L(\tau_0)\otimes\Delta(\tau^\dagger)^{(1)})$ for $\tau,\tau^\dagger\in X^+,\tau_0\in X^+_{1,p}$. Again, $\ch\Delta(\tau^\dagger)$ is a sum of at most $C(\Phi)$ characters $\chi_{KL}(\sigma)=\ch(L(\sigma_0)\otimes\Delta(\sigma^\dagger)^{(1)}$, for $\sigma,\sigma^\dagger\in X^+$ and $\sigma_0\in X^+_{1,p}$.
It follows that $\tau_0 +p\sigma_0+p^2\sigma^\dagger$ is a highest weight of a composition factor of $\Delta(\nu)$  and of $Q_{e+1}(\nu)$, and so has $e_p$-value at most $e+1$. Hence $e_p(\sigma^\dagger)\leq e-1$.
 By induction, $\text{length}(\Delta(\sigma^\dagger))\leq C^\diamond(\Phi,e-1)$. Applying the Steinberg tensor product theorem,
it follows that $\text{length}(\Delta(\nu))\leq C(\Phi)^2C^\diamond(\Phi,e-1)$. The claim follows with $C^\diamond(\Phi,e)
=C(\Phi)^{2e+1}$ for all $e\geq 0$.

We now prove the lemma by (a new) induction on $f\geq 1$. If $f=1$,  $C^\flat(\Phi,f)=C(\Phi)$ works as already remarked.  So fix
$f>1$ and write $Q_f(\nu)=Q_1(\nu_0)\otimes Q_{f-1}(\nu^\dagger)^{(1)}$. By induction, $Q_{f-1}(\nu^\dagger)$ has length
bounded by $C^\flat(\Phi,f-1)$, and hence has a $\Delta$-filtration with at most $C^\flat(\Phi,f-1)$-terms $\Delta(\tau)$
with $e_p(\tau)\leq f-1$. But $Q_1(\nu_0)$ has length at most $C(\Phi)$ with composition factors $L(\sigma_0)\otimes L(\sigma^\dagger)^{(1)}
\cong L(\sigma_0)\otimes\Delta(\sigma^\dagger)^{(1)}$, $\sigma_0\in X^+_{1,p},\sigma^\dagger\in X^+$. For $\alpha\in\Pi$, $(\sigma^\dagger,\alpha^\vee)\leq (\sigma^\dagger,\alpha^\vee_0)\leq
2h-2$, so that, independently of $p$, the possible $\sigma^\dagger$ have the form $\sigma^\dagger=\sum_i a_i\varpi_i$,
with each $a_i\leq 2h-2$. Therefore, there is an integer $M$ (given by the Weyl dimension formula) bounding
all possible $\dim\Delta(\sigma^\dagger)$ (and independent of $p$).
By
Lemma \ref{boundinglemma} (and the first paragraph of this proof),  for $\tau$ as above,  $\Delta(\sigma^\dagger)\otimes\Delta(\tau)$ has a $\Delta$-filtration
with at most $\dim\Delta(\sigma^\dagger)$-sections $\Delta(\tau')$ with $e_p(\tau')\leq f$. Thus, using the previous paragraph,
 $\text{length}(Q_f(\nu))\leq C(\Phi)C^\flat(\Phi, f-1)MC^\diamond(\Phi,f).$ In other words,
 $\text{length}(Q_f(\nu))\leq
 C(\Phi)^{(f+4)(f-1)+1}M^{f-1}$. \end{proof}

\begin{lem}\label{secondLemmatoSecondMain} Let $n$ be a non-negative integer. There exists an integer $f=f(\Phi,n)$ depending only on $\Phi$ and $n$ with
the following property. If $G$ is a simple simply connected algebraic group over a field $k$ of positive
characteristic $p$ and if $V$ is any finite dimensional rational $G$-module, then $\opH^n(G,V^{(s)})\cong \opH^n(G,V^{(s')})$ for integers $s,s'\geq f(\Phi, n)$.\end{lem}

\begin{proof}    Let $\alpha_{\text{max}}=\sum n_i\alpha_i$ be the maximal root in $\Phi^+$ and
let $c=\text{\rm max}\{n_1,\cdots n_{\text{rk}(G)}\}$ be the maximal coefficient. Let $t(\Phi)$ be the torsion exponent of $X/Q$. For an integer $m$, let $e(m):=\left[\begin{smallmatrix} m-1 \\ p-1\end{smallmatrix}\right]$, where $[\,\,]$ is
largest integer function. Set $f(\Phi,n):=e(ct(\Phi)n)+1$. Then \cite[Thm. 6.6, Cor. 6.8]{CPSK} shows that, for $s,s'\geq f(\Phi,n)$,
the cohomology spaces $\opH^n(G,V^{(s)})$ and $\opH^n(G,V^{(s')})$ are isomorphic.
\end{proof}

We now prove Theorem \ref{SecondBigTheorem} by induction on $m$. If $m=0$, take  $c(\Phi,e,0)=1$ for
all $e$. So suppose $m>0$ and the theorem holds for smaller $m$. Our proof is modeled on the $p\geq 2h-2$ case, so we
consider that case first.

We  first bound $\dim\Ext^m_G(L(0),L(\nu))$ for all primes $p\geq 2h-2$ and all $\nu\in X^+$. If $\nu=0$, then $\Ext^m_G(L(0),L(\nu))=0$,
so assume that $\nu\not=0$. By Lemma \ref{secondLemmatoSecondMain}, we can assume that the first nonzero term in the
$p$-adic expansion of $\nu$ occurs with $\nu_r\not=0$, for
$r\leq f(\Phi,m)$. Form the short exact sequence $0\to R_{r+1}(0)\to Q_{r+1}(0)\to L(0)\to 0$ in $G$-module.
 Then $\Ext^\bullet_G(Q_{r+1}(0),L(\nu))=0$,
so that $\Ext^m_G(L(0),L(\nu))\cong\Ext^{m-1}_G(R_{r+1}(0),L(\nu))$. Thus, by induction, plus Lemma \ref{firstlemmatoSecondMainThm},  $$\begin{aligned}\dim \Ext^m_G(L(0),L(\nu)) &=\dim\Ext^{m-1}_G(R_{r+1}(0),L(\nu)) \\ &\leq
C^\flat(\Phi,f(\Phi,m))-1)c(\Phi,m-1,f(\Phi,m)+1).\end{aligned}$$

Continuing with $p\geq 2h-2$, consider $L(\lambda)$ with $e_p(\lambda)\leq e$.
Form the short exact sequence $0\to R_{e+1}(\lambda)\to Q_{e+1}(\lambda)\to L(\lambda)\to 0$.
Applying the Hochschild-Serre spectral sequence, we find that
$\Ext^m_G(Q_{e+1}(\lambda),L(\nu))=0$ unless $\Hom_{G_{e+1}}(Q_{e+1}(\lambda),L(\nu))\not=0$.
In this later case,
 $\lambda=\nu_0+p\nu_1+\cdots + p^e\nu_{e}$ and $\Ext^m_G(Q_{e+1}(\lambda),L(\nu))\cong\Ext^m_G(L(0), L(\nu')),$
where $\nu'=(\nu-\lambda)/p^{e+1}$. So
$$\begin{aligned}\dim\Ext_G^m(L(\lambda),L(\nu))\leq (C^\flat(\Phi,e)-1)& c(\Phi,m-1,e+1)\\ & + C^\flat(\Phi,f(\Phi,m)-1)c(\Phi,m-1,f(\Phi,m)+1),\end{aligned}$$
competing the proof of Theorem \ref{SecondBigTheorem} for $p\geq 2h-2$.

Finally, it suffices now to give (by induction on $m$) a bound $c(\Phi,m,e)$ for any individual prime $p$.

We begin, as in the $p\geq 2h-2$ case treated above, by   bounding $\dim\Ext^m_G(L(0),L(\nu))$ for all $\nu\in X^+$. As before,
we can assume $\nu\not=0$ and $e_p(\nu)=f(\Phi,m)$. Set $r=f(\Phi,m)$ and let $Q(r+1,0)$ be the $G$-module guaranteed in
Corollary \ref{lastappendixcorollary} below. Then $\Ext^\bullet_{G_{r+1}}(Q(r+1,0),L(\nu))=0$, and so $\Ext^\bullet_G(Q(r+1,0),L(\nu))=0$. That is,
$Q(r+1,0)$ behaves in this respect like $Q_{r+1}(0)$ in the $p\geq 2h-2$ case. The number of composition factors $L(\omega)$, $\omega\in
X^+$, and the values $e_p(\omega)$ are all bounded as a function of $r$ for our given $p$, provided we choose a
single $Q(r+1,0)$ in Corollary \ref{lastappendixcorollary} for each $r$ and $p$. We now obtain a bound on $\dim\Ext^m_G(L(0),L(\nu))$ as in the
$p\geq 2h-2$ case, using property (2) of Corollary \ref{lastappendixcorollary}, namely, the fact that $L(0)$ is a $G$-quotient of $Q(r+1,0)$. The number
$C^\flat(\Phi,f(\Phi,m))$ is simply replaced by the length of $Q(r+1,0)$, and $c(\Phi,m-1,f(\Phi,m)+1)$ is replaced by
$c(\Phi,m-1,s)$, where $s$ is the maximum value of all $e_p(\omega)$ with $L(\omega)$ a composition factor of $Q(r+1,0)$.

For $0\not=\lambda\in X^+_{e+1,p}$, the argument is not as close to the $p\geq 2h-2$ case, but still uses the modules
from Corollary \ref{lastappendixcorollary}. Put $e'=(e+1)+[\log_p(2h-2)]$ with
$[-]$ the greatest integer function. Any weight $\gamma\in X^+$ of $L(-w_0\lambda)\otimes L(\lambda)$ satisfies,
for all $\alpha\in\Pi$,
$$(\gamma,\alpha^\vee)\leq (\gamma,\alpha^\vee_0)\leq(2(p^{e+1}-1)\rho,\alpha_0^\vee)=(p^{e+1}-1)(2h-2)<p^{e'+1}.$$
So $e_p(\gamma)\leq e'$. If $\Hom_{G_{e'+1}}(L(\lambda),L(\nu))=0$, then
$\Ext^\bullet_{G_{e'+1}}(Q(e'+1,\lambda),L(\nu))=0$ and $\Ext^\bullet_G(Q(e'+1,\lambda),L(\nu))=0$. In this case,
$\dim\Ext^n_G(L(\lambda),L(\nu))$ is bounded as in the $\lambda=0$ case, noting there are only finitely many $\lambda\in X^+$ with
$e_p(\lambda)\leq e$ for any fixed $e$ and $p$. Hence, we may assume that $\Hom_{G_{e'+1}}(L(\lambda),L(\nu))\not=0$. Thus,
$\nu=\lambda+p^{e'+1}\nu'$, for some $\nu'\in X^+$, and $L(\nu)=L(\lambda)\otimes L(\nu')^{(e'+1)}$. We have
$$\Ext^m_G(L(\lambda),L(\nu))=\Ext^m_G(L(0),L(-w_0\lambda)\otimes L(\lambda)\otimes L(\nu')^{(e'+1)}).$$
As noted above, all composition factors $L(\gamma)$ of $L(-w_0\lambda)\otimes L(\lambda)$ satisfy $e_p(\gamma)\leq e'$, so that
$L(\gamma)\otimes L(\nu)^{(e'+1)}$ is irreducible. Thus, the number of composition factors of $L(-w_0\lambda)\otimes L(\lambda)
\otimes L(\nu)^{(e'+1)}$ is just the number of composition factors of $L(-w_0\lambda)\otimes L(\lambda)$ in this
($\nu=\lambda+p^{e'+1}\nu'$) case. Since there are only finitely many $\lambda$ with $e_p(\lambda)\leq e$ for a given $e$ and $p$,
we obtain a bound from the $\lambda=0$ case already treated the previous paragraph. This completes the proof of Theorem \ref{SecondBigTheorem}.

\begin{rems} (a) For some readers, the cohomology case $\lambda=0$ in Theorem \ref{SecondBigTheorem} may be the most interesting.
However, our proof for that case requires also treatment of the nonzero $\lambda\in X^+$.

(b) For $m=1$, the restriction $e_p(\lambda)\leq e$ may be removed, and a bound independent of $e$ given; see Theorem \ref{MainExtOne theorem}. We do not know
if this can be done for $m>1$.\end{rems}

\medskip

Next, we prove Theorem \ref{Gsum}.
First, assume that $p\geq 3h-3$ and that the LCF holds for all $\lambda\in X^+_{1,p}\cap X^+_{\text{reg},p}$. For $\lambda_0\in X^+_{1,p}$, the proof of Lemma \ref{firstlemmatoSecondMainThm} shows that $Q_1(\lambda_0)$ has a composition series with at
most $C(\Phi)$-terms $L(\nu_0+p\nu^\dagger)=L(\nu_0)\otimes\Delta(\nu^\dagger)^{(1)}$, where $\nu_0\in X_{1,p}^+$ and $\nu^\dagger\in X^+$ satisfies
$(\nu^\dagger,\alpha^\vee)\leq 2h-2$ for all $\alpha\in\Pi$. Furthermore, there is an integer $M$ such that
$\dim\Delta(\nu^\dagger)\leq M$ for all such $\nu^\dagger$ and all primes $p$.
 Therefore,
$(\rad_G Q_1(\lambda_0))\otimes \Delta(\lambda^\dagger)^{(1)}$ has a filtration with $C(\Phi)-1$ sections $L(\nu_0)\otimes (\Delta(\nu^\dagger)^{(1)}
\otimes\Delta(\lambda^\dagger))^{(1)}$. On the other hand, each $\Delta(\nu^\dagger)\otimes\Delta(\lambda^\dagger)$ itself
has a $\Delta$-filtration with at most $M$ terms. It follows $(\rad_G Q_1(\lambda_0))\otimes \Delta(\lambda^\dagger)^{(1)}$
 has a $p$-filtration with at most $M(C(\Phi)-1)$ sections of the form $\Delta^p(\xi)$.

For any $\tau\in X^+$, $\Delta^p(\tau)$ has head $L(\tau)$. Thus,
$$\begin{aligned} \sum_{\nu,\nu_0\not=\lambda_0}\!\!\!\dim\!\Hom_G\left((\rad_G Q_1(\lambda_0))\otimes\!\Delta(\lambda^\dagger)^{(1)},L(\nu)\right) &\leq
\dim\text{\rm head}\!\left((\rad_G\! Q_1(\lambda_0))\otimes\Delta(\lambda^\dagger)^{(1)}\right)\\
&\leq M(C(\Phi)\!-1).\end{aligned}
$$

However, if $\nu_0\not=\lambda_0$, a Hochschild-Serre spectral sequence argument shows that $\Ext^1_G(Q_1(\lambda_0)\otimes\Delta(\lambda^\dagger)^{(1)},L(\nu))=0$. This vanishing also holds if $L(\nu)$ is replaced by $L(I):=\bigoplus_{\nu\in I}L(\nu)$,
where $I$ is any finite set of dominant weights $\nu$ with $\nu_0\not=\lambda_0$. (We could even take $I$  to be infinite.)
Observe that $(\rad_G Q_1(\lambda_0))\otimes\Delta(\lambda^\dagger)^{(1)}\subset \rad_G(Q_1(\lambda_0)\otimes\Delta(\lambda^\dagger)^{(1)})$,
while $Q_1(\lambda_0)\otimes\Delta(\lambda^\dagger)^{(1)}/(\rad_G Q_1(\lambda_0))\otimes\Delta(\lambda^\dagger)^{(1)}\cong L(\lambda_0)\otimes
\Delta(\lambda^\dagger)^{(1)}$ is, as a $G_1$-module, a direct sum of copies of $L(\lambda_0)|_{G_1}$. The same then holds for
$\rad_G(Q_1(\lambda_0)\otimes\Delta(\lambda^\dagger)^{(1)}/(\rad_G Q_1(\lambda_0))\otimes\Delta(\lambda^\dagger)^{(1)}$, so that there are
no non-trivial $G$-homomorphisms of this quotient module to $L(I)$. Therefore, there is a containment
$$\Hom_G\left(\rad_G(Q_1(\lambda_0)\otimes\Delta(\lambda^\dagger)^{(1)}),L(I)\right)\hookrightarrow
\Hom_G\left((\rad_G Q_1(\lambda_0))\otimes\Delta(\lambda^\dagger)^{(1)},L(I)\right),$$
so that
\begin{equation}\begin{aligned}\label{lastline}\sum_{\nu\in I}\dim\Ext^1_G(L(\lambda),L(\nu)) &=\dim\Ext^1_G(L(\lambda),L(I))\\
&\leq\dim\Hom_G\left(\rad_G(Q_1(\lambda_0)\otimes\Delta(\lambda^\dagger)^{(1)}),L(I)\right)\\
&\leq\dim\Hom_G\left((\rad_G Q_1(\lambda_0))\otimes\Delta(\lambda^\dagger)^{(1)},L(I)\right)\\
&\leq M( C(\Phi)-1).\end{aligned}\end{equation}
Since $I$ may include any finite set of weights $\nu$ with $\nu\not=\nu_0$, the theorem follows in this
large $p$ case, using $(C(\Phi)-1)M$ for $\widetilde C(\Phi)$.

It remains to treat the finitely many primes $p$ for which the assumptions above do not hold, i.~e., either $p<3h-3$
or $p\geq 3h-3$ or the LCF does not hold. In
   place of $Q_1(\lambda_0)$, use the $Q(1,\lambda_0)\in G$--mod defined in Corollary \ref{lastappendixcorollary}. The
   $G$-modules $Q(1,\lambda_0)$
have a filtration with sections $L(\nu_0)\otimes L(\nu^\dagger)^{(1)}$, the latter a homomorphic image of $L(\nu_0)\otimes \Delta(\nu^\dagger)^{(1)}$ for $\nu_0\in X^+_{1,p}$. Now consider $Q(1,\lambda_0)\otimes\Delta(\lambda^\dagger)^{(1)}$.
Observe $\Delta(\nu^\dagger)\otimes\Delta(\lambda^\dagger)$ has a $\Delta$-filtration with sections of the form
$\Delta(\xi^\dagger)$ with the total number of sections bounded by the maximum $M$ of the $\dim\Delta(\nu^\dagger)$.  Thus,
the tensor product $Q(1,\lambda_0)\otimes \Delta(\lambda^\dagger)^{(1)}$ (as well as $(\rad_{G_1} Q(1,\lambda_0))\otimes
\Delta(\lambda^\dagger)^{(1)})$ has a filtration with at most $M$ sections, all homomorphic
images of modules $\Delta^p(\xi)$. Now we can argue as above, using the fact that $Q(1,\lambda_0)$ has $L(\lambda_0)$ as a
$G$-homomorphism image.  (The role of $(\rad_GQ_1(\lambda_0))\otimes \Delta(\lambda^\dagger)^{(1)}$ is played by
$(\rad_{G_1} Q(1,\lambda_0))\otimes\Delta(\lambda^\dagger)^{(1)}$, while the role of $\rad_G(Q_1(\lambda_0)\otimes
\Delta(\lambda^\dagger)^{(1)})$ is played by the kernel of the map $Q(1,\lambda_0)\otimes\Delta(\lambda^\dagger)^{(1)}\to
L(\lambda).$). This completes the proof of Theorem \ref{Gsum}.

\begin{rems}\label{concludingremarks} (a) Theorem \ref{Gsum} fails if the $\nu_0\not=\lambda_0$ condition is dropped. Just take $\lambda=0$ and $\nu$ any nonzero
weight with $\opH^1(G,L(\nu))\not=0$. By \cite{CPS-2},
$$\Ext^1_G(L(0),L(\nu)\cong \opH^1(G, L(\nu))\hookrightarrow \opH^1(G, L(p\nu))\hookrightarrow \opH^1(G,L(p^2\nu))\hookrightarrow\cdots.$$
That is,
$\Ext^1_G(L(\lambda),L(p^m\nu))\not=0$ for all $m\geq 0$.

(b) A similar theorem holds, with essentially the same proof, if ``$L(\nu)$" is replaced in the statement by ``$\nabla(\nu)$" or with
``$\rnabla(\nu)$".

(c) According to \cite[Thm. 7.1]{CPSK}, the vector space $H^1(G,L(\mu))$ is isomorphic to $H^1(G,L(\mu'))$ if $\mu=p\mu'$, except possibly when $G$ has type $C$, $p=2$, and $\mu'\not\equiv 0$ mod $2$. Hence, we may always replace $\mu$ with a weight $\lambda$ (of the form $p^{-r}\mu$) with
$H^1(G,L(\mu))$ isomorphic to $H^1(G,L(\lambda))$ and $\lambda\not\equiv 0$ mod $p^2$. At this point, with $\lambda\not\equiv 0$ mod $p^2$, we claim there are
only finitely many $\lambda$ with $H^1(G,L(\lambda))\not=0$. If $G$ is not of type $C$, with $p=2$, then $\lambda\not\equiv 0$ mod $p$ and Theorem \ref{Gsum} applies, or else $\lambda=p\lambda'$, and Theorem 5.4 applies to $H^1(G,L(\lambda'))\cong
H^1(G,L(\lambda))$. In the exceptional case, where $G$ has type $C$ and $p=2$,
we can use Lemma \ref{Lemma52} if the rank of $G$ is greater than 2. (For rank 2, $H^1(G,L(\lambda))$ is
known \cite{Sin}.) In fact, passing to type $B$ we can, in the notation of Lemma \ref{Lemma52}, replace $\lambda$ with $\widetilde\lambda^{(1)}$ which
satisfies $\widetilde\lambda^{(1)}\not\equiv0$ mod $p^2$. Thus, there are only finitely many $\widetilde\lambda^{(1)}$
with $H^1(G',L(\widetilde\lambda^{(1)})\not=0$. Consequently, there are only finitely
many $\lambda$ with $H^1(G,L(\lambda))\not=0$.
 \end{rems}

We conclude with some connections to finite group cohomology. The notion of generic cohomology was first defined in the split case
in \cite{CPSK} and then extended to the twisted groups by Avrunin \cite{Avr}. In the split case, for any finite
dimensional rational $G$-module $V$ and positive integer
$m$, the generic cohomology of $V$ in homological degree $m$ is defined to be the common limit%
\begin{equation}
\label{generic}\opH^{m}_{\text{\textrm{\gen}}}(G,V):=\underset{d\to\infty}\lim
\opH^{m}(G(p^d),V)= \underset{s\to\infty}\lim \opH^{m}(G,V^{(s)}).
\end{equation}
The generic cohomology in the non-split case is defined similarly; see \cite{Avr} and \cite[\S7]{CPS7}.
the first limit is realized for $d$ sufficiently large (ostensibly depending on $V$), and the second limit is realized for $s$ sufficiently
large. A sufficiently large $s$ depending only on $m$ and $\Phi$ is described in Lemma \ref{secondLemmatoSecondMain}.  For $1$-cohomology ($m=1$), this
dependence on $V$ can be avoided by appealing  to Bendel, Nakano, and Pillen \cite{BNP}, and an upper
bound on $\opH^1(G(p^d),L(\lambda))$ can be obtained as in \cite{CPS7}. However, for $m>1$, such results
are unavailable. In particular, though
 $\opH^m(G(p^d),L(\lambda))$ is eventually bounded for large $d$
 by a constant depending only on the root system $\Phi$, we do not know if there exists such a universal
bound (still depending on $\Phi$) when $d$ and $L(\lambda)$ are allowed to vary.

\section{Appendix} Throughout this section, let $G$ be as in \S2(1). Thus, the characteristic of the algebraically
closed field $k$ is denoted by $p$.

The following lemma generalizes slightly \cite[Lem. 4.2]{Donkin} which treats the case $m=1$. We prove it in
general by reducing to that case. Note that $m=0$ is allowed. For $r\in\mathbb N$, write $\St_r=L((p^r-1)\rho)\in G$--mod and $\St=\St_1$.
\begin{lem}\label{Donkin}Let $M_1,M_2\in G$-mod be finite dimensional and let $m\geq 0$ be an integer. Then, for all $r$ sufficiently large,
\begin{equation}\label{Donkiniso}
\Hom_G(M_1\otimes\St_r^{(m)},M_2\otimes\St_r^{(m)})\cong\Hom_{G_{r+m}}(M_1\otimes\St_r^{(m)},M_2\otimes\St_r^{(m)})\end{equation}
via the natural restriction map.
\end{lem}
\begin{proof} Replace $M_2$ by $M_1^*\otimes M_2$ to assume that $M_1=k$. Since $\St_r^{(m)}$ is self-dual, the left-hand side of
(\ref{Donkiniso}) can then be rewritten as
$\Hom_G(\St_r^{(m)}\otimes\St_r^{(m)},M_2)$, with a similar rearrangement on the right-hand side. On both sides, we may
replace $M_2$ with its submodule $M_2^{G_m}$ of $G_m$-fixed points. As a $G$-module,
$M_2^{G_m}=M^{(m)}$, where $M$ is the $G\cong G/G_m$-module $M_2^{G_m}$. The lemma thus reduces to showing,
for  $r\gg 0$, that
$\Hom_G(\St_r^{(m)}\otimes\St_r^{(m)},M^{(m)})\cong\Hom_{G_{r+m}}(\St_r^{(m)}\otimes \St_r^{(m)},M^{(m)}),$
viz., $\Hom_G(\St_r\otimes\St_r,M)\cong\Hom_{G_{r}}(\St_r\otimes \St_r,M)$. Also, this is equivalent to
the $m=1$ case, $\Hom_G(\St_r^{(1)}\otimes\St_r^{(1)},M^{(1)})\cong\Hom_{G_{r+1}}(\St_r^{(1)}\otimes\St_r^{(1)},M^{(1)})$,
which is shown in \cite[Lemma 4.2]{Donkin}, to be an isomorphism for large $r$ (with $M^{(1)}$ replaced by any finite
dimensional $G$-module).\end{proof}

\begin{thm}Let $e\geq 1$ be an integer. There exists an integer $N=N(\Phi,e)$ with the following
property. For $\lambda\in X^+_{e,p}$, if $n\geq N$, the $G_{n+e}$-mod injective hull $Q_{n+e}(p^e(p^n-1)\rho +\lambda)$
 of $L(\lambda)\otimes\St_n^{(e)}$ has a compatible rational $G$-module structure.
 Moreover, for sufficiently large $N$, one such $G$-structure is that of the indecomposable tilting module $T((p^{n+e}-1)\rho+ (p^e-1)\rho+w_0\lambda)$ with highest weight
 $(p^{n+e}-1)\rho+ (p^e-1)\rho+w_0\lambda$.
\end{thm}

\begin{proof}We give the proof for the case $e=1$, leaving the modifications for the general case to the reader.
 If $p\geq 2h-2$, we can take $N=0$. For the remaining primes it suffices to provide an $N$ which works for
 any fixed prime $p$ and
fixed $\lambda\in X^+_{1,p}$. Define $\lambda':=(p-1)\rho+w_0\lambda$. For any finite dimensional
vector space $Y$ over $k$, let $\nu_n(Y):=[\dim Y/\dim \St_n]$, where $[-]$ is the largest integer function.
Since $\nu_n(Y)\in{\mathbb N}$, it has a minimal value over all choices of $n$ and $G$-quotients $Y$ of
$\St\otimes L(\lambda')\otimes \St_n^{(1)}$ for which the composite
\begin{equation}\label{injection}
 L(\lambda)\otimes\St_n^{(1)}\hookrightarrow \St\otimes L(\lambda')\otimes\St_n^{(1)}\cong\St_{n+1}\otimes L(\lambda')\twoheadrightarrow Y
\end{equation}
is injective. The left-hand injection is induced by the inclusion $L(\lambda)\hookrightarrow\St\otimes L(\lambda'))$ in $G$--mod. Fix  a pair $(n,Y)$  which achieves this minimum value. Since $\St_{n+r}\cong \St_n\otimes\St_r^{(n)}$, it follows (after
applying $-\otimes\St_r^{(n+1)}$ to (\ref{injection})) that, for
any integer $r\geq 0$, $(n+r,Y\otimes\St_r^{(n+1)})$ also achieves the minimum value.

Let $\overline Y$ is the quotient of $Y$ by any $G$-submodule
not containing the image of $L(\lambda)\otimes\St_n^{(1)}$. Also, $\nu_p(\overline Y)\leq \nu_p(Y)$, and $\overline Y$
fits into a diagram
like (\ref{injection}), so $\nu_p(\overline Y)=\nu_p(Y)$. There is such a quotient $\overline Y$ of minimum dimension, and we
henceforth replace $Y$ with $\overline Y$. As a result, the
$G_{n+1}$-socle of $Y$ is now homogeneous. In fact, $\soc_{G_{n+1}}Y\cong L(\lambda)\otimes \St_n^{(1)}\otimes M^{(n+1)}$, where $M\in G$-mod.
Explicitly, $M\cong \Hom_{G_{n+1}}(L(\lambda)\otimes\St_n^{(1)},Y)^{(-n-1)}$ as a $G/G_{n+1}\cong G$-module.

We claim $M$ is the trivial module $k=L(0)$.  Certainly, $k\subseteq M$, and $\dim\Hom_G(k,M)=1$,
since $\dim\Hom_G(L(\lambda)\otimes\St^{(1)}_n,\St_{n+1}\otimes L(\lambda'))=1$. Similarly (replacing
$n$ by $n+r$), we have
$\dim\Hom_G(\St_r,M\otimes\St_r)=1$ for all $r\geq 0$. (Otherwise, $Y\otimes \St_r^{(n+1)}$ contains the $G$-submodule $L(\lambda)\otimes (\St_r\otimes M)^{(n+1)}$ which contains at least two copies
of $L(\lambda)\otimes\St^{(1)}_{n+r}$ in its socle. One of these can then be factored out to give a pair $(n+r,Y')$ with
with $\nu_p(Y')<\nu_p(Y)$. This would contradict the minimality of the pair $(n,Y)$.)
For $r$ sufficiently large, we have also
$1=\dim\Hom_{G_r}(\St_r,M\otimes\St_r)$ by Lemma \ref{Donkin}.
Of course, $M\otimes\St_r\cong \St_r\oplus ((M/k)\otimes\St_r)$ in $G_r$-mod. If $M/k\not=0$, let
$E$ be an irreducible $G_r$-submodule. We can assume that $r$ is large enough that all $G_r$-composition factors
of $M$ have highest weights which are $p^r$-restricted,
and so all $G_r$-composition factors of $M$, such as $E$, belong to $G$-mod. The $G_r$-module map $E\otimes \St_r\to
M\otimes\St_r$ is $G_r$-split, hence (by Lemma \ref{Donkin}), the map
$$E\otimes\St_r\otimes\St_s^{(r)}\to M\otimes\St_r\otimes\St_s^{(r)}$$
is a $G$-map, and is $G$-split for all $s\gg 0$ (depending on $r\gg 0$). However, setting $q=\dim\St$, so that $q^n=\dim\St_n$,
we have
$$\begin{aligned}
&\nu_{n+r+s} ( \frac{Y\otimes\St^{(n+1)}_r \otimes\St_s^{(n+r+1)}}{L(\lambda)\otimes\St_n^{(1)}\otimes E^{(n+1)}\otimes
\St_r^{(n+1)}\otimes\St_s^{(n+1+r)}}) \\& =[(\dim Y-q^n\dim L(\lambda)\otimes E^{(n+1)})q^{r+s})/q^{n+r+s}]\\
&= [\dim Y/q^n -\dim L(\lambda)\otimes E^{(n+1)}]\\
& <[\dim Y/q^n]=\nu_n(Y).
\end{aligned}
$$
This contradicts the minimality of $\nu_n(Y)$. So $M/k=0$, proving the claim.

Thus, $\soc_{G_{n+1}}Y\cong L(\lambda)\otimes\St_n^{(1)}$. Since $\St\otimes L(\lambda')\otimes\St^{(1)}_n\cong
\St_{n+1}\otimes L(\lambda')$ is $G_{n+1}$-injective,
$$L(\lambda)\otimes\St^{(1)}_n\subseteq Q\subseteq \St\otimes L(\lambda')\otimes\St^{(1)}_n$$
where $Q:=Q_{n+1}(\lambda+ p(p^n-1)\rho)$ is the $G_{n+1}$-injective hull of $L(\lambda)\otimes \St^{(1)}_n$. The $G_{n+1}$-submodule
$Q$ must map injectively to the $G$-quotient $Y$ of $\St\otimes L(\lambda')\otimes\St_n^{(1)}$. Since $Q$
is $G_{n+1}$-injective, there is a $G_{n+1}$-isomorphism $Y\cong Q\oplus X$, for some $X\in G_{n+1}$-mod. However,
$\soc_{G_{n+1}}Y\cong L(\lambda)\otimes\St^{(1)}_n\cong\soc_{G_{n+1}}Q$, so $\soc_{G_{n+1}}X=0$. Thus, $X=0$,
and $Q=Y$ has a $G$-structure.

While this achieves one $G$-structure on the $G_{n+1}$-injective hull of $L(\lambda)\otimes\St_n^{(1)}$, we may have
to take $n$ larger to get the last assertion, regarding a tilting module $G$-structure. Temporarily, put
$\lambda^{\prime\prime}=(p^{n+1}-1)\rho+(p-1)\rho+w_0\lambda$, and let $T(\lambda^{\prime\prime})$ be the indecomposable
$G$-tilting module having highest weight $\lambda^{\prime\prime}$. By \cite[II.E.8]{Jan}, $T(\lambda^{\prime\prime})|_{G_{n+1}}$
is injective. Therefore, $T(\lambda^{\prime\prime})$ has a filtration as a $G_{n+1}$-module with sections
``baby Verma modules" $\widehat Z_{n+1}(\mu)$ for $G_{n+1}T$ (in the notation of \cite[II.9]{Jan}). Since $T(\lambda^{\prime\prime})$ has a unique maximal weight, namely, $\tau_0:=(p^{n+1}-1)\rho + (p-1)\rho+w_0\lambda$, it can be assumed (using
\cite[II,9.8]{Jan}) that bottom section of the filtration is $\widehat Z_{n+1}(\tau_0)$. On the other hand,
$\widehat Z_{n+1}(\tau_0)$ has $G_{n+1}$-socle $L(2(p^{n+1}-1)\rho-\tau_0)^*|_{G_{n+1}}$ \cite[II.9.6]{Jan}. But
$L(2(p^{n+1}-1)\rho-\tau_0)^*\cong L(\lambda)\otimes\St_n^{(1)}$ as  $G_{n+1}$-modules. Thus, $(L(\lambda)\otimes\St^{(1)}_n)|_{G_{n+1}}$ is contained in the $G_{n+1}$-socle
of $T(\lambda^{\prime\prime})$, so there is a $G_{n+1}$-split injection $Q\hookrightarrow T(\lambda^{\prime\prime})$.
Tensoring with $\St_u^{(n+1)}$ for $u\gg 0$ and applying Lemma \ref{Donkin}, we obtain a $G$-split $G$-module
injection
$$Q\otimes\St_u^{(n+1)}\hookrightarrow T(\lambda^{\prime\prime})\otimes\St_u^{(n+1)}.$$
By \cite[II.E.9]{Jan}, the right-hand $G$-module is a tilting module, so the left-hand side is one also. The theorem
now follows, after replacing $N$ by $N+u$ and $n$ by $n+u$.

\end{proof}

\begin{rem} Let $\lambda\in X^+_{e,p}$ and suppose that $n$ is large enough so that $Q_{n+e}(\lambda+p^e(p^n-1)\rho)
\cong T((p^{n+e}-1)\rho +(p^e-1)\rho+w_0\lambda)|_{G_{n+e}}$. Using the main result of \cite{Donkin} (generalized from
$G_1$ to $G_r$, $r\geq 1$, using Lemma \ref{Donkin}),  {\it any} two compatible $G$-structures on $Q_{n+e}(\lambda+p^e(p^n-1)\rho)$
become isomorphic in $G$-mod, after tensoring with $\St_r^{(n+e)}$ for $r\gg 0$. Also,
$Q_{n+e}(\lambda+p^e(p^n-1)\rho)\otimes\St_r^{(n+e)}\cong Q_{n+e+r}(\lambda+ p^e(p^{n+r}-1)\rho)$ (using \cite[Lemma,\S2]{Donkin}, the fact that $Q_{n+e}(\lambda+p^e(p^n-1)\rho)$ is a $G$-module and hence a $G_{n+r+s}$-module, and an Hochschild-Serre spectral
sequence argument),
it follows that any $G$-module structure on $Q_{n+e}(\lambda+p^e(p^n-1)\rho)$ becomes isomorphic, after tensoring with
$\St_r^{(n+e)}$,
 to the tilting module
$T((p^{n+r+e}-1)\rho+(p^e-1)\rho+w_0\lambda)$.
\end{rem}

A noted conjecture of Donkin \cite[(2.2)]{Donkin2} states that, for any characteristic $p$ and positive integer $e$,
if $\lambda\in X_{e,p}^+$, then $Q_e(\lambda)\cong T(2(p^e-1)\rho+w_0\lambda)|_{G_e}$. The conjecture is true if $p\geq 2h-2$
(and in some small rank examples).
One interesting feature of our theorem above is that the restriction to $G_e$ of the $G_{n+e}$-injective hull $Q_{n+e}(\lambda+p^e(p^n-1)\rho)$ is a direct sum of copies of the $G_e$-injective hull $Q_e(\lambda)$ of $L(\lambda)$. In this way, we have obtained a ``stable" version of Donkin's conjecture. We record this in the following corollary.

\begin{cor}\label{directsumcorollary} Let $e\geq 1$ be an integer and let $\lambda\in X^+_{e,p}$.
Then there is a positive integer $M$ such that $Q_e(\lambda)^{\oplus M}\cong T|_{G_e}$ for some $G$-module $T$. Moreover, $T$ can be chosen
to be the indecomposable tilting module $T((p^{n+e}-1)\rho +(p^e-1)+w_0\lambda)$ and $M=\dim \St_n$ for any sufficiently
large $n$.\end{cor}
\begin{proof} We can take $T=T((p^{n+e}-1)\rho+(p^e-1)\rho+w_0\lambda)$ as in the statement
 of the theorem. Then $T|_{G_{n+e}}$ identifies with the injective hull in $G_{n+e}$-mod of the irreducible module $L(\lambda)\otimes
 \St_n^{(e)}$. Therefore, $\soc_{G_{n+e}}T=L(\lambda)\otimes \St_n^{(e)}$. Let $Z\in G_{n}$-mod be so that
 $Z^{(e)}=\Hom_{G_e}(L(\lambda),T)$ in $G_{n+e}$-mod. It follows the $G_{n}$-socle of $Z$ must be $\St_n$. Since $\St_n$ is an injective
 $G_n$-module, $\St_n$ divides $Z$, and so $Z=\St_n$. Therefore, the $G_e$-socle of $T$ is $L(\lambda)^{\oplus M}$, for
 $M=\dim\St_n$, so that
 $T|_{G_e}\cong Q_e(\lambda)^{\oplus M}$, as required.
 \end{proof}
There is another useful way to choose a $G$-module isomorphic to a direct sum of copies of $Q_e(\lambda)$, as in Corollary \ref{directsumcorollary}. We state this result as a separate corollary. Observe that $Q(e,\lambda)$ below, when restricted
to $G_e$, is necessarily a direct sum of copies of $Q_e(\lambda)$ by properties (1) and (2).

\begin{cor}\label{lastappendixcorollary} Let $e\geq 1$ be an integer, and let $\lambda\in X^+_{e,p}$. Then there is a (rational,
finite dimensional) $G$-module $Q(e,\lambda)$
such that:

(1) $Q(e,\lambda)|_{G_e}$ is injective and projective.

(2) $L(\lambda)$ is both a $G$-submodule and a $G$-quotient module of $Q(e,\lambda)$.

(3) All irreducible $G_e$-submodules or irreducible $G_e$-quotient modules of $Q(e,\lambda)$ are isomorphic to
$L(\lambda)$.
\end{cor}
\begin{proof} Just take $Q(e,\lambda)=T\otimes \St^{(e)}_{n}$ in the proof of Corollary \ref{directsumcorollary}.\end{proof}

We  believe that many more $G_e$-modules $Q$ have the property that, for some positive integer $M$, $Q^{\oplus M}$ is the
restriction to $G_e$ of a rational $G$-module. We hope to provide necessary and sufficient conditions in a later paper.

\end{document}